\documentclass[reqno]{amsart}
\usepackage{amssymb,amscd,verbatim, amsthm,graphics, color,latexsym,amsmath,multicol}
\usepackage{fancybox}
\usepackage{longtable}

\usepackage[all]{xy}

\newcommand \fk[1]{{{\mathfrak #1}}}
\newcommand \C[1]{{\mathcal #1}}

\newcommand \wti[1]{{\widetilde {#1}}}

\newcommand\fg{\mathfrak g}

\newcommand\Ad{{\rm {Ad}}}

\def\i{^{-1}}

\newcommand \bC{{\mathbb C}}
\newcommand \bF{{\mathbb F}}
\newcommand \bH{{\mathbb H}}
\newcommand \bR{{\mathbb R}}
\newcommand \bZ{{\mathbb Z}}

\newcommand \bQ{{\mathbb Q}}

\newcommand\CB{{\C B}}
\newcommand\CC{{\C C}}

\newcommand\CO{{\C O}}

\newcommand\bbq{{\overline{\bQ_l}}}

\newcommand\ep{{\epsilon}}

\newcommand\om{{\omega}}
\newcommand\al{{\alpha}}

\newcommand\fh{{\mathfrak h}}

\newcommand\bfr{{\mathbf r}}

\def\ge{\geqslant}
\def\le{\leqslant}

\def\b{\beta}

\def\d{\delta}

\def\e{\epsilon}

\def\s{\sigma}

\def\l{\lambda}

\def\x{\xi}

\newtheorem{theorem}{Theorem}[section]

\newtheorem{corollary}[theorem]{Corollary}
\newtheorem{lemma}[theorem]{Lemma}
\newtheorem{proposition}[theorem]{Proposition}

\theoremstyle{definition}
\newtheorem{definition}[theorem]{Definition}
\newtheorem{remark}[theorem]{Remark}
\newtheorem{example}[theorem]{Example}

\newcommand\Hom{\operatorname{Hom}}
\newcommand\End{\operatorname{End}}
\newcommand\Ind{\operatorname{Ind}}

\newcommand\tr{\operatorname{tr}}

\newcommand\im{\operatorname{im}}

\newcommand\triv{\mathsf{triv}}
\newcommand\sgn{\mathsf{sgn}}
\newcommand\refl{\mathsf{refl}}
\newcommand\St{\mathsf{St}}

\newcommand\id{\mathsf{id}}
\newcommand\reg{\mathsf{reg}}

\newcommand\ev{\mathsf{even}}
\newcommand\odd{\mathsf{odd}}

\newcommand\Pin{\mathsf{Pin}}

\newcommand\dell{{\delta\mathsf{-ell}}}
\newcommand\gen{\mathsf{gen}}
\newcommand\Sg{\mathsf{Sg}}
\newcommand\DP{\mathsf{DP}}
\newcommand\qDP{\mathsf{qDP}}
\newcommand\Irr{\mathsf{Irr}}

\newcommand\sol{\mathsf{sol}}
\newcommand\gdim{\mathsf{gdim}}
\newcommand\gEP{\mathsf{gEP}}
\newcommand\gch{\mathsf{gch}}

\newcommand\Ext{\operatorname{Ext}}

\newlength{\tabwidth}
\newlength{\tabheight}
\newlength{\tabrule}
\newlength{\tabwidthx}
\newlength{\tabheightx}

\def\gentabbox#1#2#3#4{\vbox to \tabheight{\setlength{\tabrule}{#3}%
  \setlength{\tabwidthx}{#1\tabwidth}\addtolength{\tabwidthx}{\tabrule}%

\setlength{\tabheightx}{#2\tabheight}\addtolength{\tabheightx}{-\tabheight}%
  \hbox to #1\tabwidth{%
    \hspace{-0.5\tabrule}\rule{\tabrule}{#2\tabheight}\hspace{-\tabrule}%
    \vbox to #2\tabheight{\hsize=\tabwidthx%
      \vspace{-0.5\tabrule}\hrule width\tabwidthx height\tabrule%
      \vspace{-0.5\tabrule}\vfil%
      \hbox to \tabwidthx{\hss#4\hss}%
        \vfil\vspace{-0.5\tabrule}%
      \hrule width\tabwidthx height\tabrule\vspace{-0.5\tabrule}}%
    \hspace{-\tabrule}\rule{\tabrule}{#2\tabheight}\hspace{-0.5\tabrule}}%
  \vspace{-\tabheightx}}}
\def\genblankbox#1#2{\vbox to \tabheight{\vfil\hbox to
#1\tabwidth{\hfil}}}
\def\tabbox#1#2#3{\gentabbox{#1}{#2}{0.4pt}{\strut #3}}

\catcode`\:=13 \catcode`\.=13 \catcode`\;=13 
\catcode`\>=13 \catcode`\^=13
\def:#1\\{\hbox{$#1$}}
\def.#1{\tabbox{1}{1}{$#1$}}
\def>#1{\tabbox{2}{1}{$#1$}}
\def^#1{\tabbox{1}{2}{$#1$}}
\def;{\genblankbox{1}{1}\relax}
\catcode`\:=12 \catcode`\.=12 \catcode`\;=12 
\catcode`\>=12 \catcode`\^=7

\numberwithin{equation}{subsection}

\begin{document}

\title[Green polynomials]{Green polynomials of Weyl groups, elliptic
  pairings, and the extended Dirac index}

\author{Dan Ciubotaru}
        \address[D. Ciubotaru]{Dept. of Mathematics\\ University of
          Utah\\ Salt Lake City, UT 84112}
        \email{ciubo@math.utah.edu}

\author{Xuhua He}
\address[X. He]{Dept. of Mathematics and Institute for Advanced Study\\Hong Kong University of Science
  and Technology\\Clear Water Bay, Kowloon, Hong Kong}
\email{maxhhe@ust.hk}

\begin{abstract}
We provide a direct connection between Springer theory, via Green
polynomials, the irreducible representations of the pin cover $\wti
W$, a certain double cover of the Weyl group $W$, and an extended Dirac operator
for graded Hecke algebras. Our approach leads to a new and uniform
construction of the irreducible genuine $\wti
W$-characters. In the process, we  give a
  construction of the action by an outer automorphism of the Dynkin
  diagram on the cohomology groups of Springer theory, and we also introduce a $q$-elliptic
pairing for $W$ with respect to the reflection representation $V$.
These constructions are of independent interest. The $q$-elliptic pairing is a generalization of the elliptic pairing of $W$ introduced by Reeder, and it is also related to S. Kato's notion of (graded) Kostka systems for the semidirect product $A_W=\bC[W]\ltimes S(V)$. 
\end{abstract}

\thanks{The authors thank George Lusztig for his suggestions regarding
  $\d$-quasidistinguished nilpotent elements and Zhiwei Yun for his
  suggestions regarding the $W_\#$ action on the cohomology of
  Springer fibers. The authors also thank Syu Kato for helpful
  discussions about Kostka systems, and Roman Bezrukavnikov for pointing out the reference \cite{BM}. This research was supported in part
  by the NSF grant DMS-0968065 and by the HKRGC grant 602011.}

\maketitle 

\section{Introduction}

\subsection{} Graded affine Hecke algebras were defined by Lusztig
\cite{L1} in his study of representations of reductive
$p$-adic groups and Iwahori-Hecke algebras.
A Dirac operator $\C D$ for graded affine Hecke algebras
was defined in \cite{BCT}, and, by analogy with the setting of
Dirac theory for $(\fg,K)$-modules of real reductive groups, the
notion of Dirac cohomology was introduced. The Dirac cohomology and the Dirac index in
the Hecke algebra setting were further studied in \cite{COT,CT}. The
Dirac cohomology spaces are representations for a certain double cover
(``pin cover'') $\wti W$ of the Weyl group $W$. The irreducible
representations of $\wti W$ had been classified case by case in the
work of Schur, Morris, Read and others, see for example \cite{Mo,Rea,St}. Recently, it was remarked in \cite{C} (again case by
case) that there is a close relation between the representation theory of $\wti W$ and the geometry of the nilpotent cone in semisimple Lie algebras $\fg$.

\smallskip

In this paper, we provide a direct link between:
\begin{enumerate}
\item[(a)] the Springer $W$-action on
  cohomology groups and an extension of
  it to a $W\rtimes\langle\delta\rangle$-action, for the automorphism
  $\delta$ given by the action of the long Weyl group element;
\item[(b)] the irreducible representations of $\wti W$, and
\item[(c)] an extended Dirac index for tempered modules of graded Hecke algebras.
\end{enumerate}
 In particular,
our approach leads to a new, uniform construction of the
irreducible genuine $\wti W$-characters, see Theorem \ref{t:intro-2} below. The starting point is a
reinterpretation of the Lusztig-Shoji algorithm in terms of certain
$q$-elliptic pairings for $W$ with respect to the reflection
representation $V$ (see section \ref{sec:2.1}); here $q$ is an
indeterminate. This is a generalization of the elliptic pairing of $W$
introduced in \cite{Re}, see also \cite{OS,CT,COT}. The $q$-elliptic
pairing is also related to Kato's notion of (graded) Kostka systems
for the semidirect product $A_W=\bC[W]\ltimes S(V)$ and the graded Euler-Poincar\' e pairing. 

\subsection{} Let $G$ be a complex semisimple Lie group with Lie algebra $\fg$ and Weyl group $W$. Let $\C N$ be the set of nilpotent elements in $\fg$ and $\C N^\sol$ be the subset consisting of those nilpotent elements whose connected centralizers $Z_G(e)^0$  are solvable. For every element $e \in \C N$, let $A(e)=Z_G(e)/Z_G(e)^0$ be the component group of its centralizer and $\widehat{A(e)}_0$ be the set of $A(e)$-representations of Springer type. For $\phi\in \widehat{A(e)}_0$, we consider certain $q$-graded representations of $W$, denoted $X_q(e,\phi)$, see section \ref{sec:2.2}, defined using the Springer action \cite{Spr} on cohomology groups $H^{*}(\C B_e)^\phi=\Hom_{A(e)}[\phi,H^*(\C B_e)].$ By analyzing  the Lusztig-Shoji algorithm \cite{L5,Sh}, we prove first:

\begin{theorem}\label{t:intro-1} Let $e, e' \in \C N$. 
\begin{enumerate}
\item If $G \cdot e \neq G \cdot e'$, then $X_q(e,\phi)$ and $X_q(e',\phi')$ are orthogonal with respect to the $q$-elliptic pairing on $W$, for all $\phi\in\widehat{A(e)}_0$, $\phi'\in\widehat{A(e')}_0.$
\item The map $X_q(e,\phi)\to \phi$ is an isometry with respect to the $q$-elliptic pairing of $W$, and a certain $(q,M)$-elliptic pairing of $A(e)$, where $M$ is the $q$-graded $A(e)$-representation defined in (\ref{e:define-M}). When $e\in \C N^\sol$,  $M$ is the natural representation of $A(e)$ on the space of complex characters of the central torus in $Z_G(e)^0.$ 
\end{enumerate}
\end{theorem}
The case $q=1$ was known before from \cite{Re}, where it was obtained by different methods. 

The main application to $\wti W$-representations is the specialization
$q=-1$. This case is related to an action on $H^{*}(\CB_e)$ of the
extended group $W_\#=W\rtimes \langle\delta\rangle$, where $\delta$ is
the automorphism of $G$ corresponding to $-w_0$. Here $w_0$ is the longest
Weyl group element. We define this action in section \ref{w3}, by
extending the Springer action, and relate it to the results in
\cite{BS,HS,Sh} to obtain the following theorem.

\begin{theorem}\label{t:intro-delta} For every $e\in \C N$ and $i\ge 0,$
$$\tr(\delta w, H^{2i}(\C B_e))=(-1)^i\sgn(w_0)\tr(w_0w,H^{2i}(\C B_e)).$$
\end{theorem}

The main technical difficulty that one needs to bypass is that $\delta$ does not immediately
act on $\C B_e$ since $\delta(e)\neq e$ in general. Theorem \ref{t:intro-delta} turns out to be also intrinsically related to the
$\delta$-extended trace and the index of
a  Dirac operator on tempered modules of the extended graded
Hecke algebra $\bH_\#=\bH\rtimes\langle\delta\rangle$, see section
\ref{sec:twistedirac}, in particular Theorem
\ref{t:delta-twisted}. 
The extended Dirac operator and its index are
natural complements to the   case studied in
\cite{BCT,COT,CT}, and when $w_0$ is not central, they provide more
information. For example, when $G=PGL(n),$ the only tempered
$\bH_n^A$-module with nonzero Dirac index is the Steinberg module $\St$.
On the other hand, the tempered modules with nonzero extended Dirac
index are those of the form $\Ind_{\bH_J^A}^{\bH^A_n}(\St)$,
where $\bH_J^A=\prod_{j=1}^k\bH^A_{m_j}$, with $\sum_{j=1}^km_j=n$ and
all $m_j$ distinct.

\subsection{}
We explain next the main results concerning $\wti W$-representations
that we obtain via this approach. 

 Let $R(\wti W)$ be the (complexification of the) Grothendieck group of finite dimensional $\wti W$-representations, and $R(\wti W)_\gen$ the subspace spanned by the genuine irreducible $\wti W$-representations, i.e., those which do not factor through to $W$. Let $V$ be the reflection representation of $W$. Let $C(V)$ be the Clifford algebra of $V$ with respect to a $W$-invariant inner product $(~,~)$, and let $\C S$ be the unique simple spin $C(V)$-module (when $\dim V$ is even), respectively the sum of the two simple spin $C(V)$-modules (when $\dim V$ is odd). 
For every $e \in \C N$ and every $\phi\in \widehat{A(e)}_0$, set
\begin{equation}\label{Sigma-tilde}
\wti\Sigma(e, \phi)=X_{-1}(e, \phi)\otimes \C S,
\end{equation}
 By definition, this is a (virtual) character in $ R(\wti W)_\gen$, self
dual under tensoring with $\sgn$.  Theorem \ref{t:intro-delta} implies that
\begin{equation}\label{trace-X-1}
\tr(w,X_{-1}(e, \phi))=(-1)^{\dim\C B_e}\sgn(w_0)\tr(ww_0\delta,H^*(\C
B_e)^\phi).
\end{equation}
We also remark that 
$\wti\Sigma(e, \phi)$ depends only on the image of $\phi$ in $$\overline
R_{-1}(A(e))=R_{-1}(A(e))/\text{rad}\langle~,~\rangle^{-1}_{A(e)}.$$ To
see this, denote 
\begin{equation}
W_{\text{(-1)-ell}}=\{w\in W: {\det}_V(1+w)\neq 0\}.
\end{equation}
As we explain in section \ref{sec:4}, the radical of the
$(-1)$-elliptic form $\langle~,~\rangle^{-1}_W$ on $R(W)$ can
be identified with the space of characters supported on the complement
$W\setminus W_{\text{(-1)-ell}}.$ On the other hand, $\C S$ is
supported precisely on the preimage of $ W_{\text{(-1)-ell}}$ in $\wti
W$, by (\ref{spin-squared}), thus tensoring with $\C S$ kills the
radical of $\langle~,~\rangle^{-1}_W$ on $R(W)$. Because of this, it
makes sense to use the notation $\wti\Sigma(e,[\phi]),$ where $[\phi]$
is the image of $\phi$ in  $\overline R_{-1}(A(e)).$
Moreover, we may extend the definition of $\wti\Sigma(e,[\phi])$
linearly with respect to $\overline R_{-1}(A(e))$ and thus talk about
$\wti\Sigma(e,[\chi])$ for an arbitrary element $[\chi]\in \overline
R_{-1}(A(e))$ in the span of $\widehat{A(e)}_0$. The main results of section \ref{sec:5}  may be
summarized as follows. Set $a_V=1$, if $\dim V$ is even, and $a_V=2$, if $\dim V$ is odd.

\begin{theorem}\label{t:intro-2} \ 
\begin{enumerate}
\item Let $e \in \C N$  and $\phi\in \widehat{A(e)}_0$ be given. The
  character $\wti\Sigma(e,[\phi])\neq 0$ if and only if $e\in
  \C N^\sol$. In this case, $\wti\Sigma(e,[\phi])$ is the character of
  a genuine $\wti W$-representation.
\item Every genuine irreducible $\wti W$-character $\wti\sigma$ occurs in a
   $\wti\Sigma(e,[\phi])$, $e \in \C N^\sol$. Moreover, the $G$-orbit
   of $e$ is uniquely determined by $\wti\sigma.$ 
\item If $e \in \C N^\sol$, then for all $[\chi],[\chi']$, $$\langle\wti\Sigma(e,[\chi]),\wti\Sigma(e,[\chi'])\rangle_{\wti W}=a_V\langle[\chi],[\chi']\rangle_{A(e)}^{-1}.$$  
\end{enumerate}
\end{theorem}
We calculate the structure of the spaces $\overline R_{-1}(A(e))$, $e \in \C N^\sol$, and we refine part (3) of Theorem \ref{t:intro-2} in Appendix \ref{a:component}. We show that for every $e \in \C N^\sol$, there exists
 an orthogonal basis  $\{[\chi_1],\dots,[\chi_k]\}$ of $\overline
  R_{-1}(A(e))$ such that 
\begin{equation}
\wti\tau(e,[\chi_j])=\frac 1{a_e}\wti\Sigma(e,[\chi_j]), \qquad \text{
  here $a_e$ is a certain power of $2$,}
\end{equation}
is either  an
irreducible ($\sgn$ self dual) $\wti W$-character or the sum of two ($\sgn$
dual to each other) irreducible $\wti W$-characters, with two
interesting exceptions: one family of
nilpotent orbits in type $D_n$ and one orbit in $E_7$ (see Appendix
\ref{a:component}, particularly Remark \ref{r:explain}). 

\smallskip

The proof of Theorem \ref{t:intro-2} is independent of the results of
\cite{C}, and in particular, together with Corollary \ref{Casimir-action},  it recovers uniformly \cite[Theorem
1.0.1]{C}. It is also independent of the previous known
classifications, e.g., \cite{Mo,Rea,St}. Our
proof relies on Theorem \ref{t:intro-1} with $q=-1$, and on Theorem
\ref{t:intro-delta} and its relation with the extended Dirac index in
section \ref{sec:twistedirac}, together with ``Vogan's conjecture''
\cite[Theorem 4.2]{BCT}.

\subsection{}
We mention two applications of our results.

By comparing Theorem \ref{t:intro-2} and Appendix \ref{a:component} with \cite{C}, the characters
  $\wti\tau(e, [\phi])$ can be easily identified in terms of the previous
  known classifications. From this point of view, (\ref{Sigma-tilde}) can immediately be interpreted to
give a character formula of $X_{-1}(e, \phi)$ on $w\in W_{\text{(-1)-ell}}$, or alternatively, using (\ref{trace-X-1}) and Theorem
\ref{t:intro-delta}, as a character
formula of $H^*(\C B_e)^\phi$ on $\delta$-twisted elliptic
conjugacy classes (see Lemma \ref{l:delta-elliptic} for the
definition). In this way, one obtains an extension of \cite[Theorem 1.1]{CT}. See section \ref{sec:char-formula} for details.

Theorem \ref{t:intro-2} can also be used to give a solution  in terms of
Kostka-type numbers  to the problem
of decomposing tensor products $\sigma\otimes\C S$, $\sigma\in\widehat
W$. See (\ref{e:tensor-decomp}) and
Corollary \ref{c:BMc}.

\subsection{}We conclude the introduction by giving a brief summary of
the structure of the paper. In section \ref{sec:2}, we recall certain
elements of the Lusztig-Shoji algorithm, and prove Theorem
\ref{t:intro-1}. In
section \ref{sec:sol}, we study the nilpotent elements with solvable
connected centralizer, i.e., $\C N^\sol.$ In particular, following a suggestion of
Lusztig, we prove that a nilpotent element $u$ is in
$\C N^\sol$ if and only if it is $\delta$-quasidistinguished, in the
sense of Definition \ref{d:twisted-quasi}. (This definition is the
natural generalization of the notion of quasidistinguished from
\cite{Re}.) 

In section \ref{w3}, we define the action of $W_\#$ on the cohomology
groups $H^*(\C B_e)$, extending the Springer action, and prove Theorem \ref{t:intro-delta}.
In section \ref{sec:4}, we relate the $(-1)$-elliptic pairing with a
$\delta$-twisted elliptic pairing, and consider the corresponding spaces of virtual
elliptic characters.

 In section
\ref{sec:twistedirac}, we introduce the extended Dirac operator for the
extended graded Hecke algebras $\bH_\#$, and 
define its index. Using Lusztig's geometric realization of
irreducible $\bH$-modules (\cite{L2,L3}), we relate the index of
tempered modules with the character formula in
Theorem \ref{t:intro-delta}. 
 In section \ref{sec:5}, we prove the results about $\wti
W$-representations, in particular, Theorem \ref{t:intro-2}. In
Appendix \ref{a:component}, we compute explicitly the spaces
$\overline R_{-1}(A(e))$ and the associated spin representations
$\wti\tau(e,[\chi])$. In
Appendix \ref{sec:3}, we present a relation between $q$-elliptic
pairings of $W$ and the Kostka systems of \cite{Ka}. 

\section{The Lusztig-Shoji algorithm and the $q$-elliptic pairing}\label{sec:2}

\subsection{}\label{sec:2.1} If $\Gamma$ is a finite group, let
$R(\Gamma)$ denote the Grothendieck group of finite dimensional
$\bC[\Gamma]$-modules. Let $\langle~,~\rangle_\Gamma$ be the character
pairing of $\Gamma.$ If $q$ is an indeterminate, set
$R_q(\Gamma)=R(\Gamma)\otimes_\bZ \bZ[q]$ and extend
$\langle~,~\rangle$ $\bZ[q]$-linearly
to $R_q(\Gamma)$.

Let $U$ be a finite dimensional $\bC$-representation of $\Gamma$ and
$\wedge^i U$ the $i$-th exterior power of $U$ viewed as a
$\Gamma$-representation. Denote
\begin{equation}
\wedge^{-q}U=\sum_{i\ge 0}(-q)^i \wedge^i U\in R_q(\Gamma).
\end{equation}
Define the $q$-elliptic product in $R_q(\Gamma)$ to be:
\begin{equation}
\langle\chi,\chi'\rangle_{\Gamma}^{q}:=\langle\chi\otimes\wedge^{-q}U,\chi'\rangle_\Gamma\in \bZ[q].
\end{equation}

The case of interest for us will be when $\Gamma$ is a Weyl group $W$
acting on the reflection representation $U=V.$

\subsection{}\label{sec:2.2}
In the rest of this section, let $\bF$ be a finite field and
let $G$ be a connected semisimple algebraic group split over $\bF$. Let $F:G\to G$ be the corresponding Frobenius map and $G^F=G(\bF)$ be the corresponding finite group of Lie type. 

We assume furthermore that the characteristic of $F$ is sufficiently large. Specialize $q$ from
the previous section to the order of finite field $\bF$. Let $e \in \C N^F$ be given, and denote $\CO_e$ the nilpotent orbit of $e$. We set $A(e)=Z_G(e)/Z_G(e)^0$. Then $F$ acts trivially on $A(e)$ and there is a
one-to-one correspondence between $G^F$-orbits in $\CO_e^F$ and
conjugacy classes in $A(e)$. 

For every $e \in \C N$ denote $\C B_e$ the variety of Borel subalgebras of $\fg$
containing $e$ and let $d_e$ be its dimension. Springer \cite{Spr} defined an action of $W$ on the cohomology
groups (with rational coefficients) $H^j(\C B_e)$. This action
commutes with the natural $A(e)$-action. Moreover, it is known that
$H^j(\C B_e)=0$ unless $j$ is even (\cite{BS,DLP,Sh}).
Set
\begin{equation}
\widehat{A(e)}_0=\{\phi\in\widehat{A(e)}: H^{2d_e}(\C B_e)^\phi\neq 0\},
\end{equation}
the set of $A(e)$-representations of Springer type. For every pair
$(e, \phi),$ $\phi\in \widehat{A(e)}_0$, let $\sigma(e, \phi)\in \widehat W$ denote the
irreducible Springer representation afforded by
$H^{2d_e}(\C B_e)^\phi.$
For latter use, encode the
Springer correspondence as the bijective map:
\begin{equation}
\Psi:  G\backslash\{(e, \phi): e \in \C N,\phi\in\widehat {A(e)}_0\}\to
\widehat W,\quad \Psi((e, \phi))=\sigma(e, \phi).
\end{equation}
Define
\begin{equation}
\begin{aligned}
X_q(e)=\sum_{i\ge 0} q^{d_e-i}H^{2i}(\C B_e)\otimes\sgn\in R_q(W),\\
X_q(e,\phi)=\Hom_{A(e)}[\phi,X_q(e)]\in R_q(W).
\end{aligned}
\end{equation}
Thus $\sigma(e, \phi)$ occurs in degree $0$ in $X_q(e,\phi).$

Define $R_q(W)^e$ to be the subspace of $R_q(W)$ spanned by $\{X_q(e,\phi): \phi\in \widehat{A(e)}_0\}.$

\subsection{} 
The Lusztig-Shoji algorithm \cite{L5,Sh} gives a solution to a matrix equation
\begin{equation}\label{e:LS}
K(q)\Lambda(q) K(q)^t=\Omega(q),
\end{equation}
where the matrices $K(q),\Lambda(q),\Omega(q)$ are square matrices of
size $\#G\backslash\{(e, \phi):~e\in \C N,\phi\in\widehat{A(e)}_0\}$
with entries in $\bZ[q]$ to be defined next. We notice from the start
that since our normalization of $X_q(e,\phi)$ is different than the
usual one, in the sense that $\sigma(e, \phi)$ has degree $0$ in $q$
rather than top degree, we will  need to adjust in the definition of
$\Lambda(q)$ below.

Fix a set $\{e\}$ of representatives of $G$-orbits in $\C N$ and for
every such $e$, the set $\{\phi\}$ of representations in $\widehat{A(e)}_0.$

Let $K(q)$ denote the upper uni-triangular matrix whose $(e, \phi)$,
$(e',\phi')$ entry is given by the graded multiplicity of
$\sigma(e, \phi)$ in $X_q(e',\phi').$

Let $\Omega(q)$ be the symmetric matrix of fake degrees. More
precisely, the  $(e, \phi)$, $(e',\phi')$ entry in $\Omega(q)$ is the
graded multiplicity of $\sigma(e, \phi)\otimes \sigma(e',\phi')$ in
$X_q(1)$. Recall that $X_q(1)$ can be identified with the graded
representation of $W$ on the space of coinvariants of $W$ in $S(V).$
Moreover with this interpretation of $X_q(1)$, a well-known identity
of Chevalley is:
\begin{equation}\label{e:Xqprod}
X_q(1)\otimes \wedge^{-q}V=p(q)\triv,\quad \text{
  where } p(q)=\prod_i(1-q^{m_i});
\end{equation}
here $m_i$ are the fundamental degrees of $W$.

The matrix $\Lambda(q)$ is block-diagonal, with one block of size
$|\widehat{A(e)}_0|$ for each $e$. To define it precisely, we need
more notation. 
For every conjugacy class
$c$ of $A(e)$, denote by $\CO_e^F(c)$ the corresponding
$G^F$-orbits. We fix an element $e_c$ for each conjugacy class $c$ of $A(e)$. For $c=\{1\}$, we may choose $e_c=e$. We choose an element $g_c \in G$ such that $g_c \cdot e=e_c$ and set $x_c=g_c \i F(g_c) \in Z_G(e)$. Then the image $\bar x_c$ of $x_c$ in $A(e)$ is contained in $c$. 

For every $\phi\in \widehat {A(e)}$, define the $G^F$-class function
$f_\phi^e:\C N^F\to\bQ,$ by
\begin{equation}
f^e_\phi(x)=\begin{cases}\tr\phi(c),&\text{ if }x\in\CO_e^F(c),\\0,
  &\text{ if }x\notin \CO_e^F.\end{cases}
\end{equation}
Define a bilinear form on $\bQ$-valued functions on $\C N^F$, by
\begin{equation}
(f,f')=\sum_{x\in \C N^F} f(x) f'(x).
\end{equation}
Define the matrix $\wti \Lambda$ whose $(e, \phi),(e',\phi')$ entry is
\begin{equation}
(f^e_\phi,f^{e'}_{\phi'})=\delta_{e, e'}\sum_{c} |\CO^F_e(c)|\tr\phi(c)\tr\phi'(c).
\end{equation}

For every $(e, \phi)$, define the function $g_\phi^e:\C
N^F\to\bQ,$ constant on $G^F$-orbits, by
\begin{equation}
g^e_\phi(x)=\frac 1{|G^F||A(e)|}\begin{cases}\tr\phi(c) |c| |Z_G(e_c)^F|, &\text{ if }x\in\CO_e^F(c),\\0,
  &\text{ if }x\notin \CO_e^F.\end{cases}
\end{equation}
It is immediate that
\begin{equation}
(f^e_\phi,g^{e'}_{\phi'})=\delta_{e, e'}\langle\phi,\phi'\rangle_{A(e)}=\delta_{e, e'}\delta_{\phi,\phi'};
\end{equation} 
in other words, $\{g^{e}_\phi\}$ is the basis dual to $\{f^e_\phi\}.$ This means
that the inverse matrix $\wti\Lambda^{-1}$ has entries 
\begin{equation}\label{e:lambdatilde-inv}
(g^e_\phi,g^{e'}_{\phi'})=\frac {\delta_{e, e'}}{|A(e)|^2|G^F|}\sum_c |Z_G(e_c)^F|
\tr\phi(c)\tr\phi'(c) |c|^2.
\end{equation}

We can regard $\wti\Lambda$ as a matrix in $q$, and denote it by
$\wti\Lambda(q)$. The relation between $\wti\Lambda(q)$ and
$\Lambda(q)$ is given by Lusztig \cite[(24.2.7)]{L5} and
Shoji \cite[section 4]{Sh}. Since we need to account here for our
normalization of the matrix $K(q)$ in section \ref{sec:2.1}, the
relation is
\begin{equation}
\Lambda(q)=q^{\dim Z_{G}(e)} \wti\Lambda(q^{-1}).
\end{equation}

\subsection{} 
Let $M(q)$ be the symmetric matrix whose entries are $\langle
X_q(e,\phi),X_q(e',\phi')\rangle_W^q.$
We relate first $M(q)$ and $\Lambda(q)$.

\begin{theorem}\label{t:2.1}
$$\Lambda(q) M(q)=p(q)\id.$$
\end{theorem}

\begin{proof}
We compute $\Lambda(q)$. We first calculate $\Omega(q)^{-1}.$ Using
(\ref{e:Xqprod}), we have $X_q(1)\otimes \sigma\otimes
\wedge^{-q}V=p(q)\sigma,$ for every irreducible $W$-representation
$\sigma$, and then:
\begin{align*}
X_q(1)\otimes\sigma\otimes\wedge^{-q}V&=\sum_{\sigma'}
\langle\sigma',X_q(1)\otimes\sigma\otimes
\wedge^{-q}V\rangle_W\sigma'\\
&=\sum_{\sigma'}\sum_{\sigma_1,\sigma_2}
\langle\sigma_1,X_q(1)\otimes\sigma\rangle_W
\langle\sigma_2,\wedge^{-q}V\rangle_W
\langle\sigma',\sigma_1\otimes\sigma_2\rangle_W \sigma'\\
&=\sum_{\sigma'}\sum_{\sigma_1,\sigma_2}\langle\sigma_1\otimes\sigma^*,X_q(1)\rangle_W\langle\sigma_2,\wedge^{-q}V\rangle_W \langle\sigma',\sigma_1\otimes\sigma_2\rangle_W\sigma'.
\end{align*}
Thus
\begin{equation*}
\sum_{\sigma_1,\sigma_2}
\langle\sigma_1\otimes\sigma^*,X_q(1)\rangle_W\langle\sigma_2,\wedge^{-q}V\rangle_W
\langle\sigma',\sigma_1\otimes\sigma_2\rangle_W=\delta_{\sigma,\sigma'} p(q),
\end{equation*}
and therefore
\begin{equation}\label{e:om-inv}
\Omega(q)^{-1}_{\sigma_1,\sigma_2}=p(q)\sum_{\sigma_3}\langle\sigma_3,\wedge^{-q}V\rangle_W\langle\sigma_2,\sigma_1\otimes\sigma_3\rangle_W.
\end{equation}
(In the calculations above, $\sigma$, $\sigma_i$, $i=1,3$, vary over
$\widehat W$.)

Next, we compute $p(q) K(q)^t\Omega(q)^{-1} K(q).$ The $(e, \phi),$
$(e', \phi')$ entry of this matrix equals:
\begin{align*}
&p(q)\sum_{\sigma_1,\sigma_2}
K(q)^t_{(e, \phi),\Psi^{-1}(\sigma_1)}\Omega(q)^{-1}_{\sigma_1,\sigma_2}
K(q)_{\Psi^{-1}(\sigma_2),(e', \phi')}\\
&=p(q) K(q)_{\Psi^{-1}(\sigma_1),(e, \phi)}\Omega(q)^{-1}_{\sigma_1,\sigma_2} K(q)_{\Psi^{-1}(\sigma_2),(e', \phi')}\\
&=\sum_{\sigma_1,\sigma_2,\sigma_3}
\langle\sigma_1,X_q(e,\phi)\rangle_W
\langle\sigma_3,\wedge^{-q}V\rangle_W
\langle\sigma_2,\sigma_1\otimes\sigma_3\rangle_W\langle\sigma_2,X_q(e',\phi')\rangle_W\\
&=\sum_{\sigma_2}\left(\sum_{\sigma_1,\sigma_3}\langle\sigma_1,X_q(e,\phi)\rangle_W\langle\sigma_3,\wedge^{-q}V\rangle_W\langle\sigma_2,\sigma_1\otimes\sigma_3\rangle_W\right)
\langle\sigma_2,X_q(e',\phi')\rangle_W\\
&=\sum_{\sigma_2}\langle\sigma_2,X_q(e,\phi)\otimes\wedge^{-q}V\rangle_W\langle\sigma_2,X_q(e',\phi')\rangle_W\\
&=\langle X_q(e,\phi)\otimes
\wedge^{-q}V,X_q(e',\phi')\rangle_W=\langle X_q(e,\phi),X_q(e',\phi')\rangle_W^q.
\end{align*}
In other words, $p(q) K(q)^t\Omega(q)^{-1} K(q)=M(q)$, and the
conclusion follows from (\ref{e:LS}).
\end{proof}

\begin{corollary}\label{c:2.2}
If $e\neq e'$, i.e., they are representatives of distinct
$G$-orbits in $\C N$, then
$$\langle X_q(e,\phi),X_q(e',\phi')\rangle_W^q=0,$$
for all $\phi,\phi'.$ Therefore, the subspaces $R_q(W)^e$ and $R_q(W)^{e'}$ are orthogonal with respect to $\langle~,~\rangle_W^q.$
\end{corollary}

\begin{proof}
This is immediate from Theorem \ref{t:2.1}, since $\Lambda(q)$ is block-diagonal.
\end{proof}

\begin{example}
Suppose $G=SL(3)$, so $W=S_3$ acting on a two dimensional reflection space $V$. Since all component group representations of Springer type are trivial, we drop them from notation. With our conventions, we have:
$X_q(3)=(1^3)$, $X_q(21)=(21)+q(1^3)$, and $X_q(1^3)=(3)+q(21)+q^2(21)+q^3(1^3).$ Then we find $$M(q)=\text{diag}(1,1-q,(1-q^2)(1-q^3)).$$
\end{example}

\begin{example}
Suppose $G=Sp(4)$. We have five Green polynomials, with our convention:
$X_q(4)=(0\times 11),$ $X_q((22),\triv)=(1\times 1)+q(0\times 11),$ $X_q((22),\sgn)=(11\times 0),$ $X_q((211))=(0\times 2)+ q(1\times 1)+ q^2(0\times 11)$, and 
$$X_q(1^4)=(2\times 0)+q(1\times 1)+q^2(11\times 0+0\times 2)+q^3(1\times 1)+q^4 (0\times 11).$$
Then we find
$$M(q)=\text{diag}\left(1,\left(\begin{matrix}1&-q\\-q&1\end{matrix}\right),1-q^2,(1-q
^2)(1-q^4)\right).$$
\end{example}

We wish to relate
$R_q(W)^e$ with $R_q(A(e))^0$, i.e., the subspace of $R_q(A(e))$
spanned by $\widehat{A(e)}_0$.

\subsection{}
To compute $M(q)$ further, 
in light of Theorem
\ref{t:2.1}, we need to consider $\Lambda(q)^{-1}=q^{\dim Z_{G}(e)}
(\wti\Lambda^{-1})(q^{-1}).$ Notice that
\begin{equation}
|G^F|(q^{-1})=\prod_{i}(1-q^{m_i})=p(q),
\end{equation}
and using (\ref{e:lambdatilde-inv}), we see that the
$(e, \phi),(e', \phi')$ entry in $\Lambda(q)^{-1}$ equals
\begin{equation}
\delta_{e,e'}\frac 1{|A(e)|^2p(q)}\sum_{c} \tr\phi(c)\tr\phi'(c)|c|^2 q^{\dim Z_{G}(e)} |Z_G(e_c)^F|(q^{-1}).
\end{equation}

The map $g \mapsto g_c g g_c \i$ gives an isomorphism from $Z_G(e)$ to $Z_G(e_c)$. Hence 
\begin{align*}
Z_G(e_c)^F &=\{g \in Z_G(e_c); F(g)=g\} \cong \{g \in Z_G(e); F(g_c g g_c \i)=g_c g g_c \i\} \\ &=\{g \in Z_G(e); F_c(g)=g\}=Z_G(e)^{F_c}. 
\end{align*}
Here $F_c=\Ad(x_c) \circ F$ is a Frobenius morphism on $Z_G(e)$. It is easy to see that the image of $Z_G(e)^{F_c}$ under the map $Z_G(e) \to A(e)$ is $Z_{A(e)}(\bar x_c)$. Hence we have the following short exact sequence
\begin{equation}\label{surj}
\xymatrix{1 \ar[r] & (Z^0_G(e))^{F_c} \ar[r] & Z_G(e)^{F_c} \ar[r] &Z_{A(e)}(\bar x_c) \ar[r] & 1}.
\end{equation}

Let $R_e$ be the unipotent radical of $Z_G(e)$ and $H_e=Z^0_G(e)/R_e$. Then $H_e$ is a connected reductive group and $$H_e^{F_c}=(Z^0_G(e))^{F_c}/R_e^{F_c}.$$

So \begin{align*} |Z_G(e_c)^F| &=|Z_G(e_c)^{F_c}|=|Z_{A(e)}(\bar x_c)| \cdot |(Z^0_G(e))^{F_c}|\\ &=\frac {|A(e)|}{|c|} |(Z^0_G(e))^{F_c}|=q^{\dim(R_e)}\frac {|A(e)|}{|c|} |H_e^{F_c}| .\end{align*}

\begin{corollary}\label{c:2.4}
The $(e, \phi),(e', \phi')$ entry in the matrix $M(q)$ is given by 
\begin{equation*}
\delta_{e,e'} \frac 1{|A(e)|}\sum_{c} \tr\phi(c)\tr\phi'(c) |c| \zeta_{e,c}(q),
\end{equation*}
with $\zeta_{e,c}(q)=q^{\dim H_e}|H_e^{F_c}|(q \i)$.
\end{corollary}

\subsection{} Now we follow the approach in \cite[section 2.9]{Ca}. 

We fix a conjugacy class $c$ of $A(e)$. Let $T$ be an $F_c$-stable maximal torus of $H_e$ contained in a $F_c$-stable Borel subgroup. Then we can define the $F_c$-action on the character group $X$ of $T$ and on $V_e=X_{\bR}$. We have that $F_c=q F_{c, 0}$ on $V$ where $F_{c, 0}$ is an automorphism of finite order. 

The group $H_e$ is a product $H_e=H' Z^0$, where $H'$ is a semisimple group and $Z^0$ is the central torus. Then $T=S Z^0$, where $S=T \cap H'$ is a maximal torus of $H'$ and $S \cap Z^0$ is finite. We have the decomposition $V_e=V_1 \oplus V_2$, where $V_1=(Z^0)_{\bR}^\perp$ is the subspace spanned by the roots and $V_2=S_{\bR}^\perp$ is a complementary subspace. Both $V_1$ and $V_2$ are stable under the action of $F_c$ and $F_{c, 0}$. Moreover, let $V_Z$ be the vector space spanned by the characters of $Z^0$. Then $V_2 \cong V_Z$ as $F_c$-vector space and 
$$|(Z^0)^{F_c}|=\text{det}_{V_2}(q-F_{c, 0})=\text{det}_{V_Z}(q-F_{c, 0}).$$ 

Let $W_e$ be the Weyl group of $H_e$. The $W_e$-invariants of the algebra of polynomial functions on $V_1$ is a polynomial ring and there exists homogeneous elements $I_1, \cdots, I_l$ of degree $d_1, \cdots, d_l$ such that $$\mathbb C[V_1]^{W_e}=\mathbb C[I_1, \cdots, I_l]$$ and $F_{c, 0}(I_i)=\ep_{c, i} I_i$, where $\e_{c, i}$ is a root of unity. Moreover, $$|H_e^{F_c}|=|Z^{F_c}| q^N \Pi_i (q^{d_i}-\e_{c, i})=q^n \text{det}_{V_Z}(q-F_{c, 0}) \Pi_i (q^{d_i}-\e_{c, i}),$$ where $N$ is the number of positive roots. 

We may reformulate the order of $H_e^{F_c}$ in the following way. 

For any $d \in \mathbb N$, let $M[d]_c$ be the complex vector space spanned by $I_i$ with $d_i=d$. Then $F_{c, 0}$ acts on $M[d]_c$ and $$|H_e^{F_c}|=q^N \text{det}_{V_Z}(q-F_{c, 0})\Pi_d \, \text{det}_{M(d)_c}(q^d-F_{c, 0}).$$

Notice that $N$ and $V_2$ are independent of the choice of $c$. Although $T$ and $I_i$ depends on the choice of $c$, the multiset $\{d_1, \cdots, d_l\}$ is the set of degree of fundamental invariants for $H_e$ and thus for any given $d$, the dimension of the vector spaces $M(d)_c$ is independent of the choice of $c$.  

\subsection{} By (\ref{surj}), each coset of $A(e)$ contains a $F$-stable element. Moreover, if $g, g' \in Z_G(e)^F$ with $g Z_G(e)^0=g' Z_G(e)^0$, then the actions of $\Ad(g) \circ F$ and $\Ad(g') \circ F$ on $V_Z$ are the same as $Z^0$ is the central torus of $H_e$. In other words, the map $$Z_G(e)^F \to \End(V_Z), \quad g \mapsto \Ad(g) \circ F \mid_{V_Z}$$ factors through a map $A(e) \to \End(V_Z)$. For any $x \in A(e)$, we denote the corresponding endomorphism on $V_Z$ by $F_x$. Then $F_x=q F_{x, 0}$, where $F_{x, 0}$ has finite order. 

For any $g, g' \in Z_G(e)^F$, $(\Ad(g) \circ F) \circ (\Ad(g') \circ F)=\Ad(g g') \circ F^2$. Notice that $F=F_1$ acts on $V_Z$ as $q \, \id$. Thus the map $x \mapsto F_{x, 0}$ gives a group homomorphism from $A(e) \to GL(V_Z)$. 

\begin{proposition}\label{p:2.6}
Let $d \in \mathbb N$. Then there exists a representation $M(d)$ of $A(e)$ such that for any $x \in A(e)$, we have $$ \text{det}_{M(d)}(\l-x)= \text{det}_{M(d)_{x_c}}(\l-F_{x_c, 0})$$ as a polynomial on $\l$. Here $x_c$ is the conjugacy class of $x$. 
\end{proposition}

\begin{proof}
We follow the notations in \cite{LS}. 

We first consider the case where $G$ is a classical group. 

If $G=GL_n$, then $A(e)=1$ and the statement is obvious. 

If $G=Sp_n$ and $e=\oplus_i J_i^{r_i}$ be a nilpotent element in $G$, where $J_i$ denotes a nilpotent Jordan block of length $i$. By \cite[Theorem 3]{LS}, $$H_e=\Pi_{i \text{ odd}} Sp_{r_i} \times \Pi_{i \text{ even}} O_{r_i}$$ and $A(e)=(\bZ_2)^k$, where $k=\sharp\{i; i \text{ even}, r_i>0\}$.  

For any $d \in \mathbb N$, let $M(d)_{\text {odd}}=\oplus_{i \text{ odd}} M(d)^{Sp_{r_i}}$ and $M(d)_i=M(d)^{O_{r_i}}$ for $i$ even. Here for any $H=Sp_{r_i}$ or $O_{r_i}$, $M(d)^H$ is a complex vector space of dimension equal to $\dim M(d)_1$ for the group $H$, i.e., the number of the degree of the fundamental invariants for $H$ which equals $d$. 

Let $$M(d)=M(d)_{\text {odd}} \oplus \oplus_{i \text{ even}} M(d)_i.$$ We define the action of $A(e)$ on $M(d)$ as follows. 

The action of $A(e)$ on $M(d)_{\text{odd}}$ is trivial. For any $i$ even with $r_i>0$, the $i$-th copy of $\bZ_2$ in $A(e)$ acts on acts trivially on $M(d)_{i'}$ unless for $i'=i$, $d=\frac{r_i}{2}$. In the latter case, if $4 \mid r_i$, then $M(d)_i$ is $2$-dimensional and the $i$-th copy of $\bZ_2$ in $A(e)$ acts on $M(d)_i$ as permutation representation; if $4 \nmid r_i$, then $M(d)_i$ is $1$-dimensional and the $i$-th copy of $\bZ_2$ in $A(e)$ acts on $M(d)_i$ as sign representation.

By \cite[Theorem 2.12]{LS}, this is the desired representation. 

The case where $G=O_n$ can be proved in the same way. 

Now we assume that $G$ is of exceptional type and the semisimple part of $H_e$ is nontrivial. 

If $A(e)=S_2=\{1, \e\}$, then $F_\e^2=\Ad(g) \circ F^2$ for some $g \in (Z_G(e)^0)^F$. By Lang's theorem, there exists $h \in Z_G(e)^0$ such that $g=x F^2(x) \i$. So $F_\e^2=\Ad(x) \circ F^2 \circ \Ad(x) \i$. Since $F^2$ acts on $M(d)_1$ as $q^2 \, \id$, $F_\e^2$ acts on $M(d)_\e$ as $q^2 \, \id$. In particular, $F_{\e, 0}$ is an automorphism on $M(d)_\e$ with $F_{\e, 0}^2=\id$. The statement holds in this case. 

If $A(e) \neq \{1\}$ or $S_2$, we only have the following cases (see \cite[Table 5.1 \& 5.2]{LS}). 

Class $D_4(a_1)$ in $E_8(q)$. Here $A(e)=S_3$, $H_e=D_4$, $M(2)$, $M(6)$ are one-dimensional trivial representations of $A(e)$ and $M(4)$ is the irreducible 2-dimensional representation of $A(e)$. 

Class $D_4(a_1)  A_1$ in $E_8(q)$. Here $A(e)=S_3$, $H_e=A_1^3$ and $M(2)$ is the permutation representation of $A(e)$. 

Class $E_7(a_5)$ in $E_8(q)$. Here $A(e)=S_3$, $H_e=A_1$ and $M(2)$ is one-dimensional trivial representation of $A(e)$. 

Class $D_4(a_1)$ in $E_7(q)$. Here $A(e)=S_3$, $H_e=A_1^3$ and $M(2)$ is the permutation representation of $A(e)$.
\end{proof}

\subsection{} We set 
\begin{equation}\label{e:define-M}
M=\wedge^{-q} V_Z \otimes (\otimes_d (\wedge^{-q^d} M(d))).
\end{equation}
 Then the action of $A(e)$ on $V_Z$ and $M(d)$ extends in a unique way to an action on $M$ and for any $x \in A(e)$,
\begin{align*}
\tr_M(x) &=\tr_{\wedge^{-q} V_Z}(x) \times \Pi_d \, tr_{\wedge^{-q^d} M(d)}(x) \\ &=\text{det}_{V_Z}(1-qx) \times \Pi_d \, \text{det}_{M(d)}(1-q^d x) \\ &=\text{det}_{V_Z}(1-qx) \times \Pi_d \, \text{det}_{M(d)_{x_c}}(1-q^d F_{x_c}) \\ &=q^{\dim H_e-N} \text{det}_{V_Z}(q \i-x) \times \Pi_d \, \text{det}_{M(d)_{x_c}}(q^{-d}-F_{x_c, 0}) \\ &=q^{\dim H_e}|H_e^{F_{x_c}}|(q \i)
\end{align*}

We define the $(q,M)$-pairing in $R_q(A(e))$ to be 
\begin{equation}
\langle\phi,\phi'\rangle_{A(e)}^{q,M}:=\langle\phi\otimes M,\phi'\rangle_{A(e)}\in \bZ[q].
\end{equation}

Thus we have proved:

\begin{theorem}\label{t:2.7}
The map $X_q(e,\phi)\to \phi$ induces a $\bZ[q]$-isomorphic isometry with respect to the
$q$-elliptic pairing in $R_q(W)^e$ and the $(q,M)$-pairing in $R_q(A(e))^0$. More precisely:
$$\langle X_q(e,\phi), X_q(e,\phi')\rangle_W^q= \langle \phi,\phi'\rangle_{A(e)}^{q,M}.$$
\end{theorem}

\section{Nilpotent elements with solvable connected centralizer}\label{sec:sol}

\subsection{} In this section, we discuss certain nilpotent conjugacy classes that will play an essential role in our study of irreducible representations of $\widetilde{W}$. 

Let $e \in \C N$. A standard (Lie) triple of $e$ is a a triple $\{e,h,f\}\subset\fg$, such that $[h,e]=2e$, $[h,f]=-2e$, and $[e,f]=h.$ Every such triple corresponds to a Lie algebra homomorphism $\varphi: sl(2)\to \fg,$ $\left(\begin{matrix}0&1\\0&0\end{matrix}\right)\mapsto e,$ $\left(\begin{matrix}1&0\\0&-1\end{matrix}\right)\mapsto h,$ $\left(\begin{matrix}0&0\\1&0\end{matrix}\right)\mapsto f$.  By the Dynkin-Kostant classification, see for example \cite[pages 35--36]{CM}, the map $\{e,h,f\}\to e$ gives a one-to-one correspondence between the set of $G$-conjugacy classes of Lie triples and $G$-orbits of nilpotent elements in $\fg$. We refer to the element $h$ as a neutral element for $e$ and when we wish to emphasize the dependence of $h$ on $e$, we denote it by $h_e$. 

\begin{definition}
An element $e \in \C N$ is called distinguished (\cite{Ca}) is the
centralizer $Z_{\fg}(e)$ does not contain any nonzero semisimple element.

An element $e \in \C N$ is called quasidistinguished if there exists
a semisimple element $t\in Z_G(e)$ such that $Z_G(t)$ is semisimple and $e$ is a distinguished element in the Lie algebra of $Z_G(t)$. 
\end{definition}

Set $u=\exp(e)$. Then $e$ is quasidistinguished if and only if $u$ is quasidistinguished in the sense of  \cite{Re}, i.e.,  there exists a semisimple element $t\in G$ such that $u \in Z_G(t)$ and $Z_G(t u)$ does not contain any nontrivial torus. Indeed, one direction is obvious. For the converse, let $g=u t$. By Jordan decomposition, any element that commutes with $g$ also commutes with $t$ and $u$. Thus $Z_G(g)=Z_{Z_G(t)}(u)$. Hence $Z_G(g)$ doesn't contain a nontrivial torus implies that $Z_G(t)$ is semisimple and $u$ is distinguished in $Z_G(t)$. Hence $e$ is distinguished in the Lie algebra of $Z_G(t)$.  

\smallskip

Recall that $\C N^\sol$ is the set of $e \in \C N$ such that $Z_G(e)^0$
is a solvable group. It is clear that every distinguished nilpotent element $e$ is also
quasidistinguished and belongs to $\C N^\sol$. It is proved in
\cite[Lemma 7.1(1)]{Re2} that if $e$ is quasidistinguished, then necessarily $e\in
\C N^\sol$. However, $e \in \C N^\sol$ does not imply that $e$ is quasidistinguished.

\subsection{}\label{sec: 4.1} Let $\delta$ be the automorphism of $G$ given by the action of $-w_0$ on the root datum, where $w_0\in W$ is the longest element. More precisely, $\delta$ is the order two automorphism of the Dynkin diagram when $G$ is of type $A_n$, $D_{2n+1}$, or $E_6$, and $\delta$ is trivial for other simple groups. 

We set $G_\#=G\rtimes\langle\delta\rangle$. Following \cite[section 9]{St1}, we call an element $g \in G_\#$ quasi-semisimple if there exists a Borel subgroup $B$ of $G$ and a maximal torus $T \subset B$ such that $g B g \i=B$ and $g T g \i=T$. In this case $Z_G(g)$ is a reductive group. Moreover, by \cite[Theorem 7.5]{St1}, if $g\in G_\#$ is semisimple, then $g$ is quasi-semisimple.

\begin{definition}\label{d:twisted-quasi}

An element $e \in \C N$ is called $\delta$-quasidistinguished if 
there exists a semisimple element $t \d \in Z_{G_\#}(e)$ such that $Z_G(t \d)$ is semisimple and $e$ is a distinguished element in the Lie algebra of $Z_G(t \d)$. 
\end{definition}


\smallskip

Suppose that $t\delta$ is semisimple in $G_\#$.
The condition that $Z_G(t\delta)$ be semisimple implies that $t\delta$ is an isolated (torsion) element of $G\delta$ in the terminology of \cite[section 2]{L-discon} or \cite[section 3.8]{Re3}. For basic results about the isolated elements, see \cite[section 2]{L-discon}, particularly \cite[Lemma 2.6]{L-discon}. The classification of isolated semisimple elements is known, and we recall it next, following \cite[sections 3.8, 4.1-4.5]{Re3}. Let $\fk t\subset \fk b$ be $\delta$-stable Cartan and Borel subalgebras, respectively. If $\Phi$ is the root system of $\fk g$ corresponding to $\fk t$, with positive roots given by $\fk b$. Call two roots $\al,\beta\in \Phi$ $\delta$-equivalent if $\al|_{\fk t^\delta},\beta|_{\fk t^\delta}$ are proportional via a positive constant. 
If $a$ is a $\delta$-equivalence class in $\Phi$, then $a$ is a $\delta$-orbit in $\Phi$, except in type $A_{2n}$, when $a$ could be of the form $\{\al,\delta(\al),\al+\delta(\al)\}$. Let 
$$\fg_a=\sum_{\al\in a}\fg_\al,\quad \gamma_a=\sum_{\al\in a}\al|_{\fk t^\delta},\ \beta_a=\frac 1{f_a}\gamma_a,$$
where $f_a=|a|$, unless $a$ is the exception in type $A_{2n}$, when $f_a=4.$

With this notation, the root-space decomposition of $\fk g^{t\delta}$, $t=\exp(x)$, $x\in \fk t^{\delta}$ is (\cite[Proposition 3.8]{Re3}):
\begin{equation}\label{e:isolated}
\fk g^{t\delta}=\fk t^\delta\oplus\sum_{a}\fk g_a^{t\delta},
\end{equation}
where the sum of over the $\delta$-equivalence classes $a\in \Phi/\delta$ such that $\langle\gamma_a,x\rangle\in \{-1,0,1\}$. Each $\fk g_a^{t\delta}$ is one-dimensional, affording either a root $\beta_a$ or $2\beta_a$, the latter case may only occur in the exceptional $a$ in $A_{2n}.$

\begin{proposition}\label{p:unip-cent} A nilpotent element $e \in \C N$
  is $\delta$-quasidistinguished if and only if $e \in \C N^\sol,$ i.e.,
  the centralizer $Z_G(e)^0$ is solvable.
\end{proposition}

\begin{proof} One can prove uniformly that ``$e$ is $\delta$-quasidistinguished'' implies ``$e \in \C N^\sol$'' analogously with the untwisted case \cite[Lemma 7.1(1)]{Re2}, as follows. Suppose $e$ is $\delta$-quasidistinguished, and let $t \d\in G \d \subset G_\#$ be a semisimple element as in Definition \ref{d:twisted-quasi}. Let $H_e$ be the (connected) reductive part of $Z_G(e)$ and let $\fh_e$ be the Lie algebra. Since $t\delta\in G_\#$ acts on $H_e$ by conjugation, one can consider $\Ad(t\delta)|_{\fh_e}:\fh_e\to \fh_e$ and let $\fh_e^{t\delta}$ be the fixed points. The algebra $\fh_e^{t\delta}$ is a reductive Lie algebra, since $t\delta$ is semisimple. However, $\fh_e^{t \d}$ does not contain nonzero semisimple element. This means that $\fh_e^{t\delta}=0.$ By \cite[Corollary 10.12]{St1}, $\fh_e$ has zero derived subalgebra, which implies $H_e$ is a torus, equivalently $Z_G(e)^0$ is solvable.

The proof of the converse direction is case by case.

\smallskip

In type $A$, we consider $GL(n)$ rather than $SL(n)$ for simplicity. The nilpotent orbits in $\C N^\sol$ are in one two one correspondence with partitions of $n$ into distinct parts via the Jordan canonical form. Let $\lambda$ be such a partition and break $\lambda$ into $\lambda_0$ containing all the even parts and $\lambda_1$ containing all the odd parts. Let $2m$ be the sum of parts in $\lambda_0$. Let $J_{2m}=\left(\begin{matrix}0&I_m\\-I_m&0\end{matrix}\right)$, where $I_m$ is the identity $m\times m$ matrix. Set $t_{m}=\text{diag}(J_{2m},I_{n-2m})$. Then $t_m \d$ is a semisimple element.

We consider the automorphism $\delta: GL(n)\to GL(n)$ given by $\delta(x)=(x^T)^{-1}.$ Thus $$Z_{GL(n)}(t_m\delta)=Sp(2m)\times O(n-2m).$$ 
Let $e_{\lambda_0}$ be a distinguished nilpotent element in $sp(2m)$ parameterized by the even partition $\lambda_0$ and $e_{\lambda_1}$ a distinguished nilpotent element in $o(n-2m)$ parameterized by the odd partition $\lambda_1$. Then $e_\lambda=\left(\begin{matrix}e_{\lambda_0}&0\\0& e_{\lambda_1}\end{matrix}\right)$ is a representative of the class in $\C N^\sol$ labeled by $\lambda$ and it is $\delta$-quasidistinguished by construction.

\smallskip

In $Sp(2n)$ (resp. $SO(2n+1)$ or $SO(4n)$), the automorphism $\delta$ is trivial, and one can see from the classification of nilpotent classes that the classes in $\C N^\sol$ are parameterized by partitions of $2n$ (resp. $2n+1$ or $4n$) where every part is even (resp. odd) and each part appears with multiplicity at most $2$. It is easy to check that every such nilpotent class is quasidistinguished. For example, suppose $\lambda=(a_1,a_1,a_2,a_2,\dots,a_k,a_k,a_{k+1},\dots,a_\ell)$ is a partition of $2n$, where $a_1<a_2<\dots<a_k<a_{k+1}<\dots<a_\ell$ are even numbers. Let $t_m\in Sp(2n)$ be a semisimple element whose centralizer is $Sp(2m)\times Sp(2n-2m),$ where $2m=\sum_{i=1}^\ell a_i.$ We choose a distinguished nilpotent element $e_1$ in $sp(2m)$ corresponding to the partition $(a_1,a_2,\dots, a_\ell)$ and a distinguished nilpotent element $e_2$ in $sp(2n-2m)$ corresponding to the partition $(a_1,a_{2},\dots,a_k)$. Then $e_\lambda=e_1\times e_2$ is a representative of the nilpotent class in $sp(2n)$ labeled by $\lambda$ and it is quasi-distinguished by construction.

\smallskip

If $G=SO(4n+2)$, $\delta$ corresponds to the automorphism of order $2$ of the Dynkin diagram. Suppose the roots of the corresponding root system of type $D_{2n+1}$ are labeled $\{\al_1,\al_2,\dots,\al_{2n+1}\}$ and that $\delta$ acts by interchanging $\al_{2n}$ and $\al_{2n+1}$ and fixes the other roots. The orbits in $\C N^\sol$ are parameterized by partitions of $4n+2$ into odd parts, each part of multiplicity at most $2$. By \cite[Lemma 2.9(iv)]{LS} and \cite[Table 4.3.1]{Gor}, there exist $n+1$ classes of involutions in $G\delta$ with representatives $t_{i-1}\delta$ having centralizers of type $B_{i-1}\times B_{2n+1-i}$, where $1\le i\le n+1$. Here $t_0=1,$ and $t_i$, $1\le i\le n$, is the order two element in the standard torus in $SO(4n+2)$ corresponding to the root $\al_i$, see for example \cite[(4.4.4)]{Gor} for the precise definition. In particular, $t_{i-1}$ commutes with $\delta$, so $t_{i-1}\delta$ is semisimple, $1\le i\le n+1$. The construction of $\delta$-quasidistinguished orbits proceeds then exactly as in the untwisted $Sp(2n)$ example above. 

\smallskip

When $G$ is exceptional of type $G_2$, $F_4,$ $E_7$, or $E_8$, the automorphism $\delta$ is trivial, and one verifies the claim from the classification of nilpotent orbits and their centralizers. This is an easy direct calculation using the explicit classification of nilpotent classes. We give the results below, but skip the details, since the calculation is very similar to the twisted $E_6$ example which we'll explain in detail.

If $G$ is of type $G_2$ or $F_4$, the only nilpotent orbits in $\C N^\sol$ are already distinguished, so there is nothing to check.

For groups of type $E$, we use the following labeling of the Dynkin diagrams (this is not the Bourbaki notation):
\begin{equation}\label{Dynkin-E6}
\xymatrix@R=4pt{
& &  & {\alpha_4} \ar@{-}[d] \\
&\alpha_1 \ar@{-}[r]&\alpha_2\ar@{-}[r] & {\alpha_3} \ar@{-}[r] & {\alpha_5} \ar@{-}[r] & {\alpha_6}\ar@{-}[r] & {\alpha_7} \ar@{-}[r] & {\alpha_8}.
}\\
\end{equation}

For types $E_7$ and $E_8$, denote $t_i=\exp(\frac 1{(\gamma,\omega_i^\vee)}\omega_i^\vee)\in T,$ where $\omega_i^\vee$ is the fundamental coweight corresponding to the $i$-th simple root, and $\gamma$ is the highest positive root. 

If $G$ is of type $E_7$, the non-distinguished nilpotent orbits in $\C
N^\sol$ are denoted in the Bala-Carter classification \cite{Ca} by
$E_6(a_1)$ and $A_4+A_1$. They come from the regular nilpotent orbits in $Z_G(t_4)=A_7$ and $Z_G(t_3)=A_3\times A_3\times A_1$, respectively. 

If $G$ is of type $E_8$, the non-distinguished nilpotent orbits in $\C
N^\sol$ are $D_5+A_2$, $D_7(a_1)$, $D_7(a_2)$, and $E_6(a_1)+A_1.$
They come from $E_7(a_4)$ in $Z_G(t_8)=E_7\times A_1$, $E_7(a_3)$ in
$Z_G(t_8)=E_7\times A_1$, the regular nilpotent orbit in
$Z_G(t_6)=D_5\times A_3$, and the regular nilpotent orbit in $Z_G(t_2)=A_7\times A_1$, respectively.

\smallskip

It remains to analyze the case $G=E_6$ and $\delta$ coming from the automorphism of order $2$ of the Dynkin diagram. There are seven nilpotent orbits in $\C N^\sol$ labeled: $E_6$, $E_6(a_1)$, $E_6(a_3)$, $D_5$, $D_5(a_1)$, $A_4+A_1$, and $D_4(a_1)$. 

Suppose that $t\delta$ is semisimple. We use (\ref{e:isolated}) to realize $\delta$-quasidistinguished nilpotent orbits. The explicit cases in $E_6$ are in \cite[section 4.5]{Re3}. For each $x\in \fk t^\delta$ such that $t=\exp(x)$ that appears (there are five cases), we compute the simple roots of the Lie algebra $\fk g^{t\delta}$ as in (\ref{e:isolated}). Then for each distinguished nilpotent element in $\fg^{t\delta}$ we match its Dynkin-Kostant diagram (in $\fk g^{t\delta}$) with a diagram in $\fg$. This is done as follows: the Dynkin-Kostant diagram of $e \in \fg^{t\delta}$ gives the values of the simple roots $\beta$ for $\fk g^{t\delta}$ on the neutral element $h_e\in\fk t^{\delta}$, and thus we can determine $h_e$. Next, one makes $h_e$ dominant with respect to the simple roots in $\fk g$ and computes the Dynkin-Kostant diagram in $\fk g.$   The explicit results are below. We denote by $\omega_i^\vee\in \fk t$ the fundamental coweight corresponding to $\al_i.$

\begin{enumerate}

\item[(0)] $x_0=0$, $\fk g^{t_0\delta}=F_4$, with simple roots:
\begin{equation*}
\xymatrix@R=2pt{
 {\alpha_4} \ar@{-}[r]&{\alpha_3}\ar@{=>}[r] 
&{\frac 12(\alpha_2+\alpha_5)} \ar@{-}[r]&{\frac
12(\alpha_1+\alpha_6)}.
}
\end{equation*}
The fixed point group $F_4$ has four distinguished nilpotent orbits: $F_4$, $F_4(a_1)$, $F_4(a_2)$, and $F_4(a_3)$ which correspond in $E_6$ to: $E_6$,  $D_5$,  $E_6(a_3)$, and  $D_4(a_1)$, respectively.

\item[(1)] $x_1=\frac 14(\om_1^\vee+\om_6^\vee)$, $\fk g^{t_1\delta}=B_3\times A_1$, with simple roots:
\begin{equation*}
\xymatrix@R=2pt{
 {\alpha_4} \ar@{-}[r]&{\alpha_3}\ar@{=>}[r] 
&{\frac 12(\alpha_2+\alpha_5)} &{\beta},
}
\end{equation*}
where $\beta=\frac 12[(\al_1+\al_2+2\al_3+\al_4+2\al_5+\al_6)+(\al_1+2\al_2+2\al_3+\al_4+\al_5+\al_6)].$ 
There is only one distinguished nilpotent orbit in $\fk g^{t_1\delta}$, the regular one, which corresponds to $D_5(a_1)$ in $E_6.$

\item[(2)] $x_2=\frac 16(\om_2^\vee+\om_5^\vee)$, $\fk g^{t_2\delta}=A_2\times A_2$, with simple roots:
\begin{equation*}
\xymatrix@R=2pt{
 {\alpha_4} \ar@{-}[r]&{\alpha_3} 
&{\frac 12(\alpha_1+\alpha_6)}\ar@{-}[r] &\beta',
}
\end{equation*}
where $\beta'=\frac 12[(\al_1+2\al_2+2\al_3+\al_4+\al_5)+(\al_2+2\al_3+\al_4+2\al_5+\al_6)]$.
There is only one distinguished nilpotent orbit in $\fk g^{t_2\delta}$, the regular one, which corresponds to $D_4(a_1)$ in $E_6.$

\item[(3)] $x_3=\frac 14 \om_3^\vee$, $\fk g^{t_3\delta}=A_3\times A_1$, with simple roots:
\begin{equation*}
\xymatrix@R=2pt{
 \beta''\ar@{-}[r] &{\frac 12(\alpha_2+\alpha_5)} \ar@{-}[r] 
&{\frac 12(\alpha_1+\alpha_6)}\ar@{-}[r]  &{\al_4},
}
\end{equation*}
where $\beta''=\frac 12[(\al_1+\al_2+2\al_3+\al_4+\al_5)+(\al_2+2\al_3+\al_4+\al_5+\al_6)].$
There is only one distinguished nilpotent orbit in $\fk g^{t_3\delta}$, the regular one, which corresponds to $A_4+A_1$ in $E_6.$

\item[(4)] $x_4=\frac 12 \om_4^\vee$, $\fk g^{t_4\delta}=C_4$, with simple roots:
\begin{equation*}
\xymatrix@R=2pt{
 {\alpha_3} \ar@{=>}[r] 
&\frac 12(\alpha_2+\alpha_5) \ar@{-}[r]&\frac
12(\alpha_1+\alpha_6)\ar@{-}[r] & \beta''',
}
\end{equation*}
where $\beta'''=\frac 12[(\al_2+\al_3+\al_4)+(\al_3+\al_4+\al_5)]$. There are two distinguished nilpotent orbits in $C_4$, the regular orbit $(8)$ and the subregular orbit $(62)$, which correspond in $E_6$ to $E_6(a_1)$ and $E_6(a_3)$, respectively.
\end{enumerate}
This finishes the proof for $E_6$ and therefore, the proof of the proposition.
\end{proof}

\

Now we analyze the action of $\delta$ on $G$-orbits of Lie triples in $\fg$.

\begin{lemma}\label{l:4.4} Let $\C O$ be a nilpotent $G$-orbit in $\fg.$ There
  exists a $\delta$-stable Levi subalgebra $\fk m$ and a Lie triple
  $\{e,h,f\}\subset \fk m$, $e\in \CO$, such that $e$ is
  $\delta$-quasidistinguished in $\fk m.$
\end{lemma}

\begin{proof}
We prove the statement case by case. When $\delta=1$, the claim is
immediate from the Bala-Carter classification, in fact, $\fk m$ can be
chosen so that $e$ is distinguished in $\fk m.$

Let $G=GL(n)$. The nilpotent orbit $\C O$ is given via Jordan form by a
partition $\lambda$ of $n$. Let $I=\{1,2,\dots,n-1\}$ be the indexing
set for the simple roots. The automorphism $\delta$ acts on $I$ as
$\delta(i)=n-i.$  We construct a subset $J\subset I$, such
that $\delta(J)=J.$ 
Start by setting $J=I$. If $\lambda$ has only
distinct parts, by Proposition \ref{p:unip-cent}, $\CO$ is
$\delta$-quasidistinguished in $\fg$, so $\fk m=\fk g.$ Suppose
$\lambda$ has equal parts. Write
$\lambda=\lambda'\sqcup\{r_1,r_1\}\sqcup\dots\sqcup\{r_\ell,r_\ell\}$,
where $\lambda'$ is the largest subset of $\lambda$ having only
distinct parts. Set $k=\sum_{j=1}^\ell r_j.$ Remove from $I$ the indices $r_1, n-r_1,
r_1+r_2,n-(r_1+r_2), \dots, k,n-k,$ the resulting subset is $J$. By
construction, $\delta(J)=J.$ Then $\fk m$ is the Levi subalgebra
corresponding to $J$, thus $\fk m=\sum_{j=1}^\ell (gl(r_j)\oplus
gl(r_j))\oplus gl(n-2k)$. Let $\{e,h,f\}\subset \fk m$ be a Lie triple
in $\fk m$
representing the regular nilpotent orbits on each factor $gl(r_j)$ and
the nilpotent orbit parameterized by $\lambda'$ on $gl(n-2k).$ The
latter is $\delta$-quasidistinguished by Proposition
\ref{p:unip-cent}.

Let $\fg=so(2n)$. Let $I=\{1,2,\dots,n\}$ be the
indexing set of simple roots, and suppose that the branch point of the
Dynkin diagram is at $n-2$, so that $\delta(i)=i$ for all $1\le i\le
n-2$ and $\delta(n-1)=n.$ Suppose $\C O$ is parameterized by a
partition $\lambda$ of $2n$. Then each even part in $\C O$ appears
with even multiplicity. Partition $\lambda$ as $\lambda=\lambda'\sqcup\{r_1,r_1\}\sqcup\dots\sqcup\{r_\ell,r_\ell\}$,
where $\lambda'$ is the largest subset of $\lambda$ such that
$\lambda'$ has only odd parts and each part occurs with multiplicity
at most $2$. As before, set $k=\sum_{j=1}^\ell r_j.$ The subset
$J\subset I$, $\delta(J)=J$, is
$$J=\{1,\dots,r_1-1\}\cup \{r_1+1,\dots,r_1+r_2-1\}\cup \{k-r_\ell+1,\dots,k-1\}\cup\{k+1,\dots,n\},$$
and the corresponding Levi  subalgebra $\fk m=\oplus_{j=1}^\ell
gl(r_j)\oplus so(2n-2k).$ Let $\{e,h,f\}$ be a Lie triple in $\fk m$
such that $e$ represents the regular nilpotent orbits on the $gl(r_j)$
factors and the $\delta$-quasidistinguished orbit parameterized by
$\lambda'$ on the $so(2n-2k)$ factor.

When $\fg$ is of type $E_6$, we list the orbits with the corresponding
$\fk m$ and $e\in \fk m$ in Table \ref{table:E6levi}. The indexing set
$I=\{1,\dots,6\}$ corresponds to the Dynkin diagram
(\ref{Dynkin-E6}). We do not include in the table the
$\delta$-quasidistinguished orbits in $E_6$: $E_6$, $E_6(a_1)$, $E_6(a_3)$,
$D_5$, $D_5(a_1)$, $A_4+A_1$, and $D_4(a_1).$

\begin{table}[h]
\caption{$E_6$: nilpotent orbits\label{table:E6levi}}
\begin{tabular}{|c|c|c|}
\hline
$\C O$  &$\fk m$ &$\delta$-quasidistinsguished $e\in \fk m$\\
\hline
$A_5$ &$\{1,2,3,5,6\}$ &$(6)$\\
$D_4$ &$\{2,3,4,5\}$ &$(7,1)$\\
$A_4$ &$\{1,2,3,5,6\}$ &$(5,1)$\\
$A_3+A_1$ &$\{1,2,3,5,6\}$ &$(4,2)$\\
$2A_2+A_1$ &$\{1,2,4,5,6\}$ &regular\\
$A_3$ &$\{2,3,5\}$ &$(4)$\\
$A_2+2A_1$ &$\{1,3,4,6\}$ &regular\\
$2A_2$ &$\{1,2,5,6\}$ &regular\\
$A_2+A_1$ &$\{1,2,3,5,6\}$ &$(3,2,1)$\\
$A_2$ &$\{2,3,4,5\}$ &$(3,3,1,1)$\\
$3A_1$ &$\{2,4,5\}$ &regular\\
$2A_1$ &$\{2,5\}$ &regular\\
$A_1$ &$\{3\}$ &regular\\
$0$ &$\emptyset$ &$0$\\
\hline
\end{tabular}
\end{table}
\end{proof}

\begin{proposition}\label{prop:4.5} Let $\CO$ be a nilpotent $G$-orbit in $\fg.$ There exists a Lie triple $\{e,h,f\}$, with $e\in \CO$, $h\in \fk t^\delta,$ and an element $g\in G$ such that  $\delta(\phi)=\Ad(g)
  \phi$, and $\delta(g)g\in Z_G(\phi)^0.$ 
\end{proposition}

\begin{proof}
Assume $\delta\neq 1$ (otherwise $g=1$). First, suppose $e$ is
$\delta$-quasidistinguished in $\fg$. By definition, there exists
$t\delta$, $t\in T^\delta$, an isolated semisimple element of
$G\delta$ such that $e\in Z_\fg(t\delta).$ We may choose the Lie triple
so that $\phi\subset Z_\fg(t\delta).$ Thus $\Ad(t\delta) \phi=\phi,$ or equivalently $\delta(\phi)=\Ad(t^{-1})\phi.$ So
we may choose $g=t^{-1}\in T^\delta$, and then $\delta(g)g=t^{-2}$. We
claim that $t^{-2}\in Z_G(\phi)^0.$ Indeed, by the proof of
Proposition \ref{p:unip-cent}, $t^2=1$ unless $\fg$ is of type $A_n$
or when $\fg$ is of type $E_6$ and $\CO$ is $D_5(a_1)$ or
$A_4+A_1$. But in all of these cases, $Z_G(\phi)$ is connected, so
there is nothing to prove.

If $e$ is not $\delta$-quasidistinguished in $\fg,$ by Lemma
\ref{l:4.4}, there exists a $\delta$-stable Levi $\fk m$ such that
$\phi\subset \fk m$ and $e$ is $\delta$-quasidistinguished in $\fk m.$
From the proof of Lemma \ref{l:4.4}, we see that whenever $\fk m$ has
two factors of the same type which are flipped by $\delta$, the two
factors are of type $A_{r-1}$ (for some $r$) and the nilpotent
elements on these factors are equal to a regular nilpotent
element $e_{r-1}$. This means that $\delta$ fixes the pair
$(e_{r-1},e_{r-1})$, and therefore by the discussion in the case $\fk
m=\fk g$, there exists $g\in \Ad(\fk m)$ such that $\delta(g)g\in
Z_M(\phi)^0=M\cap Z_G(\phi)^0.$ 
\end{proof}

\section{An action of $W_\#$ on $H^*(\C B_e)$}\label{w3}

\subsection{} Let $\d$ be the automorphism on $G$ and on $W$ defined
in section \ref{sec: 4.1} and $W_\#=W \rtimes
\langle\delta\rangle$. In this section, we construct a natural action
of $W_\#$ on $H^*(\C B_e)$, which extends the action of $W$ we
discussed in section \ref{sec:2.2}, such that the following theorem holds. 

\begin{theorem}\label{t:4.1} For every $e\in \C N$ and $0\le i\le d_e$, 
$$\tr(\delta w, H^{2i}(\C B_{e}))=(-1)^i \sgn(w_0)\tr(w_0 w, H^{2i}(\C B_{e})).$$
\end{theorem}

In particular, we have

\begin{corollary}\label{c:action-1}
For any $w \in W$, $$X_q(e,\phi)(\d w)=(-1)^{d_e}\sgn(w_0)X_{-q}(e, \phi)(w_0 w).$$ 
\end{corollary}

\subsection{} We assume that $\d \neq \id$. 

We fix a $\d$-stable Borel subalgebra $\fk b$ of $\fg$. Let $\fg_{\reg}$ be the set of regular semisimple elements in $\fg$ and $\fk b_{\reg}=\fg_{\reg} \cap \fk b$. Define the action of $B$ on $G \times \fk b$ by $b \cdot (g, b')=(g b \i, \Ad(b) b')$. Let $G \times^B \fk b$ be the quotient scheme and $G \times^B \fk b_{\reg}$ be the image of $G \times \fk b_{\reg}$ in $G \times^B \fk b$. The Springer resolution (of $\fg$) is the map 
\begin{equation}
q: G \times^B \fk b \to \fg, \quad (g, b) \mapsto \Ad(g) b. 
\end{equation}

Its restriction $q_{\reg}$ to $G \times^B \fk b_{\reg}$ gives an unramified Galois covering of $\fg_{\reg}$ whose Galois group is $W$. 

We fix a prime number $l$ invertible in $\bF$. Let $\bbq_{G \times^B \fk b}$ be the trivial sheaf on $G \times^B \fk b$ and $\bbq_{G \times^B \fk b_{\reg}}$ be the trivial sheaf on $G \times^B \fk b_{\reg}$. Set $\Psi=Rq_! \bbq_{G \times^B \fk b} [\dim(G)]$ and $\Psi_{\reg}=R (q_{\reg})_! \bbq_{G \times^B \fk b_{\reg}} [\dim(G)]$. Since $q$ is a small map, $\Psi$ is the intersection cohomology complex $IC(\fg, \Psi_{\reg})$. Thus $\End(\Psi)=\End(\Psi_{\reg}) \cong \bbq[W]$ and we have a natural action of $W$ on $\Psi$. Notice that the map $q$ is in fact $G_\#$-equivariant. Hence the automorphism $\d$ on $G$ induces an action $\d^*: \Psi \to \Psi$. We have that $(\d^*)^2=1$ and $(\d)^* w=\d(w) \d^*: \Psi \to \Psi$. 

\subsection{}\label{choice1} Let $e \in \C N$. If $g \in G$ such that $\d(e)=\Ad(g) (e)$. Then $\d(g) g \in Z_G(e)$. We choose $g \in G$ such that $\d(e)=\Ad(g) (e)$ and the image of $\d(g) g$ in $A(e)$ lies in the kernel of $\phi$ for all $\phi \in \widehat{A(e)}_0$. By Proposition \ref{prop:4.5}, such $g$ always exists. Thus $\Ad(g)^* \circ \d^*: \Psi_e \to \Psi_e$ satisfies 
\begin{equation}\label{e:5.2.1}
(\Ad(g)^* \circ \d^*)^2=(\Ad(\d(g) g))^*=\id.
\end{equation}

For any $w \in W$, we have the following commuting diagram 

\begin{equation}\label{e:5.2.2}
\xymatrix{\Psi_e \ar[r]^-{\d^*} \ar[d]_-{\d(w)} & \Psi_{\Ad(g) e } \ar[r]^-{\Ad(g)^*} \ar[d]_-w & \Psi_e \ar[d]_-w \\ \Psi_e \ar[r]^-{\d^*} & \Psi_{\Ad(g)e} \ar[r]^-{\Ad(g)^*} & \Psi_e.}
\end{equation}

From (\ref{e:5.2.1}) and (\ref{e:5.2.2}), we see that the map 
\begin{equation}\label{e:action-W-sharp}
w \mapsto w, \quad \d \mapsto \Ad(g)^* \circ \d^*
\end{equation}
 gives an action of $W_\#$ on $\Psi_e$ and hence on the cohomology of $\Psi_e$. By construction, this action commutes with the action of $A(e)$ on $\Psi_e$. In particular, for any $i \ge 0$ and $\phi \in \widehat{A(e)}_0$, we may regard $H^{2 i}(\C B_e)^\phi$ as a $W_\#$-module and thus $X_q(e, \phi)$ as a virtual character of $W_\#$. 

Notice that the $W_\#$-module structure depends on the choice of $g$. If we pick a different element $g' \in G$ such that $\d(e)=\Ad(g')(e)$ and $\d(g') g' \in Z_G(e)^0$, then $g \i g' \in Z_G(e)$ and thus the actions of $\delta$ on $H^{2 i}(\C B_e)^\phi$ (defined using $g$ and $g'$) differ by  $\phi(g \i g')$. 

\subsection{}\label{nice-choice} Now we discuss the choice of $g$
which makes the action of $\delta$ on $H^*(\C B_e)$ nice. We construct
such an element $g=g_1$ without using Proposition \ref{prop:4.5}. {A similar action of $A(e) \rtimes \langle\delta\rangle$ on $H^*(\C B_e)$ is studied by Bezrukavnikov and Mirkovi\'c in \cite{BM}.}

As before we assume $p$ and $q$ are large. We assume furthermore that $p \equiv 1 \mod 3$ if $G$ is of type $E_8$. 
Let $F:G\to G$ be the split Frobenius map. For $e \in \C N^F$, $F(\C B_e)=\C B_e$. We say that $e$ {\it split} (with respect to $F$) if all the irreducible components of $\C B_e$ are $F$-stable. By \cite[Proposition 3.3]{Sh} and \cite[Section 3]{BS}, each nilpotent orbit of $\fg$ contains exactly one split $G^F$-orbit. 

Let $F'=F\delta=\d F:G\to G$ be a twisted Frobenius map. The action of $F'$ on $W$ is conjugation by $w_0$. The following result is proved by Hotta and Springer for unitary groups in \cite[Theorem 3.1]{HS}, as a consequence of a specialization theorem, by Shoji for the other classical groups in \cite[Theorem 4.18]{Sh}, and by Beynon and Spaltenstein for exceptional groups, in particular for $E_6$ in \cite[Theorem 4.1]{BS}, as a consequence of the Lusztig-Shoji algorithm.
\begin{theorem}\label{HS-S-BS}
Let $\C O$ be a nilpotent orbit of $\fg$. There exists a bijection $\s$ from the set of $G^{F'}$-orbits in $\C O^{F'}$ to the set of $G^F$-orbits in $\C O^F$ such that 
\begin{equation*}
\tr((F')^* \circ w, H^{2i}(\C B_{e}))=(-1)^i\sgn(w_0) \tr (F^* \circ w_0 w, H^{2i}(\C B_{\s(e)})).
\end{equation*}
\end{theorem}

We assume furthermore that $\s(e)$ is split with respect to $F$. Let $g \in G$ such that $\d(e)=\Ad(g)(e)$. Then $e \in \C O^{F_0}$, where $F_0=\Ad(g \i) \circ F$ is again a split Frobenius morphism. There exists $h \in G$ such that $\Ad(h)(e)$ is split with respect to $F_0$, i.e. $F_0(\Ad(h)(e))=\Ad(h)(e)$ and all the irreducible components of $\C B_{\Ad(h)(e)}=\Ad(h) \C B_e$ are $F_0$-stable. In other words, $h \i F_0(h) \in Z_G(e)$ and all the irreducible components of $\C B_e$ are $\Ad(h \i F_0(h)) \circ F_0=\Ad(h \i) \circ F_0 \circ \Ad(h)$-stable. Set 
\begin{equation}\label{g1}
g_1=g (h \i F_0(h)) \i,
\end{equation}
 and $F_1=\Ad(g_1) \i \circ F$. Then $\d(e)=\Ad(g_1)(e)$ and $F_1$ is a split Frobenius morphism and $e$ is split with respect to $F_1$. 

By \cite[section 5(C)]{BS}, $F^*$ acts by $q^i$ on $H^{2 i}(\C B_{\s(e)})$ and $F_1^*$ acts by $q^i$ on $H^{2 i}(\C B_e)$. Thus
\begin{align*}
\tr((F')^* \circ w, H^{2i}(\C B_{e})) &=\tr(F_1^* \circ (\Ad(g_1)^* \circ \d^*) \circ w, H^{2i}(\C B_{e})) \\ &=q^i \tr((\Ad(g_1)^* \circ \d^*) \circ w, H^{2i}(\C B_{e}))
\end{align*}
and
\begin{align*}
\tr(F^* \circ w_0 w, H^{2i}(\C B_{\s(e)})) &=q^i \tr(w_0 w, H^{2i}(\C B_{\s(e)})) \\ &=q^i \tr(w_0 w, H^{2i}(\C B_{e})).
\end{align*}

By Theorem \ref{HS-S-BS}, 
\begin{equation}\label{eq-4.3.1}
\tr((\Ad(g_1)^* \circ \d^*) \circ w, H^{2i}(\C B_{e}))=(-1)^i \sgn(w_0)\tr(w_0 w, H^{2i}(\C B_{e})).
\end{equation}

In particular,  $$\tr((\Ad(g_1)^* \circ \d^*) \circ w_0, H^{2i}(\C B_{e}))=(-1)^i \sgn(w_0)\dim(H^{2i}(\C B_{e})).$$
 Since $(\Ad(g_1)^* \circ \d^*)$ commutes with $w_0$ by (\ref{e:5.2.2}),  $(\Ad(g_1)^* \circ \d^* \circ w_0)^2=\Ad(\d(g_1) g_1)^*$ acts on $H^{2 i}(\C B_e)$ via the image of $\d(g_1) g_1$ in $A(e)$, and so $\Ad(g_1)^* \circ \d^* \circ w_0$ acts on $H^{2 i}(\C B_e)$ as an element of finite order. Hence it acts by the scalar $(-1)^i \sgn(w_0)$. 

So $(\Ad(g_1)^* \circ \d^* \circ w_0)^2=\Ad(\d(g_1) g_1)^*$ acts on $H^{2 d_e}(\C B_e)$ as the identity. Therefore the image of $\d(g_1) g_1$ in $A(e)$ lies in the kernel of $\phi$ for all $\phi \in \widehat{A(e)}_0$. 

\subsection{} We proved Theorem \ref{t:4.1} over a finite
field. In order to pass from (large) characteristic $p$ to characteristic $0$,
first note that the representation of $W$ on $H^{2 i}(\C B_e)$ is
independent of the characteristic \cite[section 3]{Sp-Weyl}.
Now we
choose $g$ as in the proof of Proposition \ref{prop:4.5}. Then the
action of $\Ad(g)^* \circ \d^*$ is again independent of the
characteristic. As explained in section \ref{choice1}, there exists $z
\in A(e)$ such that for any $i$ and $\phi$, $\Ad(g)^* \circ \d^* \circ
w_0$ acts on $H^{2 i}(\C B_e)^\phi$ as $(-1)^i \sgn(w_0) \phi(z)$. In
fact, in characteristic $p$, $z$ is the image of $g_1 \i g$ in
$A(e)$. The component group $A(e)$ is independent of the choice of
characteristic.  In characteristic $0$, set $g_1=gz_1^{-1} \in Z_G(e)$,
for a representative $z_1\in Z_G(e)$ of $z$. Then $\Ad(g_1)^* \circ
\d^* \circ w_0$ acts on $H^{2 i}(\C B_e)$ as $(-1)^i \sgn(w_0)$,
and the same argument as at the end of section \ref{nice-choice} shows
that the image of $\d(g_1) g_1$ in $A(e)$ lies in the kernel of $\phi$ for all $\phi \in \widehat{A(e)}_0$. 

\begin{remark}\label{r:g1}
In the rest of the paper, unless otherwise stated, we regard $H^{2
  i}(\C B_e)^\phi$ as a $W_\#$-module via (\ref{e:action-W-sharp})
with respect to the element $g_1$.
As we discussed in section \ref{choice1}, $g_1$ is uniquely determined by  (\ref{eq-4.3.1}) up to right multiplication by $Z_G(e)^0$. 
\end{remark}

\subsection{} Suppose that $\d=1.$ We set $g_1=1.$ 

When $\fg$ is a classical Lie algebra $sp(2n)$, $so(2n+1)$, or
$so(4n)$,  Theorem
\ref{HS-S-BS} is proved in \cite[Theorem 4.18]{Sh} and therefore
(\ref{eq-4.3.1}) holds.

When $\fg$ is exceptional of type $G_2,F_4,E_6,E_7,E_8$, we do not
know an explicit reference for Theorem \ref{HS-S-BS}. However, the
argument in \cite[Theorem 4.18]{Sh} (see also \cite{BS}) can be applied in these cases as
well, as soon as we construct the correct matching $\sigma$ so that
the analogues of 
\cite[Proposition 1.12 and Lemma 4.20(ii)]{Sh} hold. This is done as
follows. 

Let $e\in \C N^F$. For every $\phi\in \widehat {A(e)}_0$, let
$\text{hdeg}(\sigma(e,\phi))$ denote the lowest harmonic degree of
$\sigma(e,\phi)$, and set $d_\phi=(-1)^{\text{hdeg}(\sigma(e,\phi))}.$ 

\begin{lemma}[{compare with \cite[Proposition 1.12]{Sh}}]
There exists a conjugacy class $c_0$ of $A(e)$ such that
$\phi(c_0)=d_\phi d_1,$ for all $\phi\in\widehat{A(e)}.$
\end{lemma}
\begin{proof}
When $\fg$ is a simple exceptional Lie algebra not of type $E_6,$ the
component group $A(e)$ can be $1$, $\bZ/2\bZ$, or $S_n$ with
$n=3,4,5.$ The claim is obvious in the case $1$ or $\bZ/2\bZ.$ In the
$S_n$ cases, by inspection of the tables in \cite[pages 429-432]{Ca},
we see that all $\sigma(e,\phi)$ have the same parity of the lowest
harmonic degrees. Thus, we may choose $c_0=1$ in these cases.
\end{proof}
Define the map $\sigma: G^F\backslash G\cdot e\to G^F\backslash G\cdot
e,$ by $\CO_e^F(c)\mapsto \CO_e^F(cc_0).$ This makes sense because in
the cases when $A(e)$ is not abelian, we chose $c_0=1.$ By inspection
of tables in \cite[Chapter 22]{LS}, we see that the analogue of
\cite[Lemma 4.20(ii)]{Sh} holds:
\begin{equation}
|Z_{G^F}(e_c)|(-q)=|Z_{G^F}(e_{cc_0})|(q).
\end{equation}

\smallskip

An alternative argument, again case by case, is as follows. The graded
$W$-representations $H^*(e)^\phi$ are explicitly computed in all the
exceptional cases: for $G_2,F_4$ in \cite{Sh2}, and for type $E$ in
\cite{BS2}.  One can check that if an irreducible $W$-representation
$\mu$ occurs in  $H^{2i}(\C B_e)^\phi$ then
$(-1)^{\text{hdeg}(\mu)}=\sgn(w_0) (-1)^i$. See also \cite[page 19,
Remark (a)]{BS2} The lowest harmonic
degrees can be read from \cite{Ca}.  Since $w_0$ is central, it acts
on every irreducible $\mu$ by $(-1)^{\text{hdeg}(\mu)},$ and
(\ref{eq-4.3.1}) follows.

\section{The twisted elliptic form}\label{sec:4}

\subsection{} Let $\ep \in \{+1, -1\}$. Set $$\overline
R_{\ep}(W)=R(W)_\bC/\text{rad}\langle~,~\rangle_W^{\ep}.$$ For any $e \in \C N$, let $R_{\ep}(W)^e$ be the image of $R(W)^e$ in $R_{\ep}(W)$. By Corollary \ref{c:2.2}, $R_{\ep}(W)=\oplus_e R_{\ep}(W)^e$, where $u$ runs over nilpotent conjugacy classes of $G$. 

\begin{proposition}\label{p:6.1} Let $e \in \C N$ be given.
\begin{enumerate}
\item $\overline R_1(W)^e\neq 0$ if and only if $u$ is
  quasidistinguished.
\item $\overline R_{-1}(W)^e\neq 0$ if and only if $u$
  is $\delta$-quasidistinguished, or equivalently, $e \in \C N^\sol$.
\end{enumerate}
\end{proposition}

\begin{proof}
If $Z_G(e)^0$ is not solvable, then $Z_G(e)^{F_c}$
contains as a subgroup the $\bF_q$-points of a rank one semisimple
group, and in particular, $|Z_G(e)^{F_c}|$ (as a polynomial in $q$) is
divisible by $(q^2-1)$. Therefore, $|Z_G(e)^{F_c}|(\ep)=0$ for $\ep \in \{+1, -1\}$. By Corollary \ref{c:2.4}, $\overline R_1(W)^e=\overline R_{-1}(W)^e=0$.

Thus we may reduce to the case where $e \in \C N^\sol$. The statements then follow by an analysis similar to the proof of Proposition \ref{p:2.6}.
\end{proof}

\subsection{} In the rest of this section, we focus on the $(-1)$-elliptic form $\langle~, ~\rangle^{-1}_W$. It is by definition
\begin{equation}
\langle\chi,\chi'\rangle^{-1}_W=\langle \chi\otimes \wedge
V,Y\rangle_W,\quad\text{ where }\wedge V=\oplus_{i\ge 0} \wedge^i V.
\end{equation}
We will relate it to the $\d$-twisted elliptic form, here
$\delta$ is the automorphism of $(W,V)$ given by $-w_0$:
\begin{equation}
\delta(\xi)=-w_0(\xi),\ \xi\in V,\quad \delta(w)=w_0ww_0,\ w\in W.
\end{equation}

The $\d$-twisted elliptic form is defined as follows. 

If $(\sigma,X)$ is a $W$-representation, let
$(\sigma^\delta,X^\delta)$ be the $\delta$-twisted representation,
i.e.,
$$\sigma^\delta(w)x=\sigma(\delta(w))x=\sigma(w_0ww_0)x,\ w\in W,\
x\in X^\delta=X.$$
We choose the intertwininer $\phi:(\sigma^\delta,X^\delta)\to (\sigma,X)$ to be $\phi(x)=\sigma(w_0)x.$ Clearly $\phi^2=1.$ Define the
$\delta$-twisted character of $\sigma$:
$$\tr_\sigma^\delta(w)=\tr(\sigma(w)\circ\phi^{-1}),$$
and using the explicit form of $\phi$, we find 
\begin{equation}\label{e:untwist}
\tr_\sigma^\delta(w)=\tr_\sigma(ww_0),\ w\in W.
\end{equation}

\begin{definition}
Let $R^\delta(W)$ denote the $\bZ$-span of $\delta$-twisted
characters $\tr_\sigma^\delta$, for all irreducible
$W$-representations $\sigma$. Define the $\delta$-elliptic pairing in
$R^\delta(W)$:
\begin{equation}
\langle\chi,\chi'\rangle_W^\dell=\frac 1{|W|}\sum_{w\in
  W}\chi(w)\chi'(w){\det}_{V}(1-w\delta).
\end{equation}
\end{definition}

\begin{lemma}\label{l:iso-twisted}
The map $\tr^\delta_\sigma\to \tr_\sigma$ induces a linear isometry
between the spaces $(R^\delta(W),\langle~,~\rangle^\dell_W)$ and
$(R(W),\langle~,~\rangle^{-1}_W)$:
$$\langle\tr^\delta_\sigma,\tr^\delta_{\sigma'}\rangle^\dell_W=\langle\tr_\sigma,\tr_{\sigma'}\rangle^{-1}_W.$$
\end{lemma}

\begin{proof}
This is a straightforward calculation:
\begin{align*}
\langle\tr^\delta_\sigma,\tr^\delta_{\sigma'}\rangle^\dell_W&=\frac
1{|W|}\sum_{w\in
  W}\tr^\delta_\sigma(w)\tr^\delta_{\sigma'}(w){\det}_V(1-w\delta)\\&=\frac
1{|W|}\sum_{w\in
  W}\tr_\sigma(ww_0)\tr_{\sigma'}(ww_0){\det}_V(1+ww_0)\\&=\frac
1{|W|}\sum_{w\in W}\tr_\sigma(w)\tr_{\sigma'}(w){\det}_V(1+w)=\langle\tr_\sigma,\tr_{\sigma'}\rangle^{-1}_W.
\end{align*}
\end{proof}

\subsection{} Set $$\overline
R^\d(W)=R(W)_\bC/\text{rad}\langle~,~\rangle_W^{\dell}.$$
Under  the
isometry in Lemma \ref{l:iso-twisted}, $\overline R_{-1}(W)$ may be
identified with $\overline R^\delta(W).$ 

We give an more explicit description of $\overline R^\d(W)$ using the $\d$-elliptic conjugacy classes. 

\begin{definition}
The $\delta$-twisted conjugacy class of $w\in W$ is
$$\CC_w=\{w'w\delta(w')^{-1}:w'\in W\}.$$ A twisted class $\CC$ is called
$\delta$-elliptic if $\CC\cap W_J=\emptyset,$ for all $\delta$-stable proper parabolic
subgroups $W_J$ of $W$. An element $w\in W$ is called
$\delta$-elliptic if it belongs to a $\delta$-elliptic conjugacy class.
\end{definition}
The following lemma is well-known, see for example \cite{He}.

\begin{lemma}\label{l:delta-elliptic}
The following are equivalent for an element $w\in W$:
\begin{enumerate}
\item $w$ is $\delta$-elliptic;
\item $\det_V(1-w\delta)\neq 0$;
\item $V^{w\delta}=0.$
\end{enumerate}
\end{lemma}

\subsection{} For every $\delta$-stable parabolic subgroup $W_J$ of $W$, let
$\delta\text{-}\Ind_{W_J}^W: R^\delta(W_J)\to R^\delta(W)$ be the induction
functor:
$\delta\text{-}\Ind_{W_J}^W(\sigma)=\bC[W]\delta\otimes_{\bC[W_J]\delta}\sigma.$
Denote
$$R^\delta_{\text{ind}}(W)=\sum
\delta\text{-}\Ind_{W_J}^W(R^\delta(W_J)),$$
where the sum is over all $\delta$-stable proper parabolic subgroups $W_J$ of 
$W.$ 

The following proposition
is a straightforward modification of \cite[Proposition (2.2.2)]{Re},
and we omit the proof.

\begin{proposition}\label{p:twisted-radical} We keep the notations as above. Then 

(1) $\text{rad}\langle~,~\rangle^\dell_W=R^\delta_{\text{ind}}(W)_\bC$. 

(2) The dimension of $\overline R^\delta(W)$ equals the
number of $\delta$-elliptic conjugacy classes in $W$.
\end{proposition}

\begin{remark}\label{r:ell-classes}
In light of Proposition \ref{p:twisted-radical}, we record the number of $\delta$-elliptic conjugacy classes in the
irreducible Weyl groups (where $\delta=-w_0$), see for example \cite[section 3.1]{Re},
\cite[section 7]{He} and \cite[section 6]{GKP}.
\begin{enumerate}
\item ${}^2A_{n-1}$: the number of partitions of $n$ into odd parts
  (this is the same as the number of partitions of $n$ into distinct
  parts);
\item $B_n$: the number of partitions of $n$;
\item $D_{2n}$: the number of partitions of $2n$ into an even number of
  parts;
\item ${}^2D_{2n+1}$: the number of partitions of $2n+1$ into an odd
  number of parts; 
\item $G_2$: 3; $F_4$: 9; ${}^2E_6$: 9; $E_7$: 12; $E_8$: 30.
\end{enumerate}
\end{remark}

\section{Extended Dirac operator}\label{sec:twistedirac}

\subsection{} We retain the notation from the previous sections. Fix a
$W$-invariant bilinear form $(~,~)$ on $V$. Then $W$ is a Weyl group in the Euclidean vector space
$V,(~,~)$ with semisimple root system $\Phi\subset V^*$, positive roots $\Phi^+$,
and simple roots $\Pi.$ Let $\bfr$ denote an indeterminate to be
specialized later. Recall the automorphism $\delta$, $\delta^2=1,$ of
the root system given by $-w_0$.

\begin{definition}[Lusztig,\cite{L1}]
The graded affine Hecke algebra $\bH$ with equal
parameters attached to $(\Phi,V,W)$  is the unique  associative $\bC[\bfr]$-algebra with identity generated by $\{\xi\in V^*_\bC\}$ and $\{w\in W\}$ such that:
\begin{enumerate}
\item $\bH\cong \bC[\bfr]\otimes_\bC\bC[W]\otimes_\bC S(V^*_\bC)$, as $(\bC[\bfr]\otimes_\bC\bC[W],S(V^*_\bC))$-bimodules;
\item $\xi\cdot s_\al-s_\al\cdot s_\al(\xi)=2\bfr\xi(\al^\vee),$ $\al\in\Pi,$ $\xi\in V_\bC^*.$
\end{enumerate}
\end{definition}
The center of $\bH$ is $\bC[\bfr]\otimes_\bC S(V_\bC^*)^W$, hence the
central characters are parameterized by $W$-orbits in $\bC\oplus V_\bC.$

Let $*$ denote the conjugate linear anti-automorphism of $\bH$ defined on generators via
\begin{equation}\label{star-oper}
\bfr^*=\bfr,\ w^*=w^{-1},\ \xi^*=w_0\cdot \delta(\xi)\cdot w_0,\ w\in W,\ \xi\in V.
\end{equation}
For every $\xi\in V,$ define
\begin{equation}
\wti\xi=\frac 12(\xi-\xi^*).
\end{equation}
It is clear that $\wti\xi^*=-\wti\xi.$ Moreover, it is known that $w\cdot \wti\xi\cdot w^{-1}=\wti{w(\xi)}.$

We consider the extended algebra $\bH_\#=\langle\bH,\delta\rangle$, and extend the $*$-operation to $\bH_\#$ by setting $\delta^*=\delta.$

\subsection{} We recall the classification of simple $\bH$-modules
from \cite{L2,L3}, in particular \cite[section 1]{L3}. We denote by
$H^{G'}_*()$ and $H_{G'}^*()$ the $G'$-equivariant
homology and cohomology as in the references.

 If $x\in \fg,$ denote $Z_{G\times
  \bC^\times}(x)=\{(g,\lambda)\in G\times \bC^\times:
\Ad(g)x=\lambda^{2}g\}.$

Let $e\in \fg$ be a nilpotent element, and $(s,r_0)$ a semisimple
element in the Lie algebra of $Z_{G\times \bC^\times}(e).$ In
particular, $[s,e]=2r_0e.$ Choose $A\subset Z_{G\times \bC^\times}(x)$
a torus containing $\exp(s,r_0)$. Let $\bC_{s,r_0}$ denote the one
dimensional $H_A^*(\{pt\})$-module obtained by evaluation at $(s,r_0)$.
(One may identify $H_A^*(\{pt\})$ with the space of polynomials $\bC[\fk
a]$, where $\fk a$ is the Lie algebra of $A$.) From (\cite[10.12]{L2}, \cite[1.13]{L3})
\begin{equation}
E_{e,s,r_0}=H_*(\C B_e^s)=\bC_{s,r_0}\otimes_{H^*_A(\{pt\})} H_*^A(\C
B_e)=\bC_{s,r_0}\otimes_{H^*_{Z_{G\times \bC^\times}(e)^0}(\{pt\})}
H_*^{Z_{G\times \bC^\times}(e)^0}(\C B_e),
\end{equation}
where $\C
B_e^s$ is the variety of Borel subalgebras of $\fg$ containing $e$ and
$s$. The space $E_{e,s,r_0}$ carries an $\bH$-action such that $\bfr$ acts by
$r_0$. 

Let $Z(e,s)=Z_{G\times \bC^\times}(e)\cap (Z_G(s)\times \bC^\times).$
Denote by $A(e)$ and $A(e,s)$ the group of components of $Z_{G\times
  \bC^\times}(e)$ and $Z(e,s)$, respectively. The natural map
$A(e,s)\to A(e)$ is an injection, so $E_{e,s,r_0}$ carries an action
of $A(e,s)$ obtained by restriction from the natural action of $A(e)$
on $H_*^{Z_{G\times \bC^\times}(e)^0}(\C B_e).$ Set
\begin{equation}
E_{e,s,r_0,\psi}=\Hom_{A(e,s)}[\psi,E_{e,s,r_0}],
\end{equation}
and $\widehat{A(e,s)}_0=\{\psi: E_{e,s,r_0,\psi}\neq 0\}.$ 
By
\cite{L2}, $\psi\in \widehat{A(s,e)}$ if and only if $\psi$ occurs in
the action of $A(s,e)$ on $H_*(\C B_e^s).$

\begin{theorem}[{\cite[Theorem 1.15]{L3}}]\label{Lusztig-irred} Let
  $r_0\neq 0.$

\begin{enumerate}
\item Let $e,s$ be as above and $\psi\in \widehat{A(e,s)}_0$. The $\bH$-module $E_{e,s,r_0,\psi}$ has a unique maximal
  submodule. Let $\overline E_{e,s,r_0,\psi}$ be the irreducible quotient.

\item The map $(e,s,\psi)\to \overline E_{e,s,r_0,\psi}$ gives a
  one-to-one correspondence between the of $G$-conjugacy classes of
  triples $(e,s,\psi)$ where $e\in\fg$ is nilpotent, $s\in \fg$ is
  semisimple with $[s,e]=2r_0e$, $\psi\in\widehat{A(s,e)}_0$, and the
  set of simple $\bH$-modules on which $\bfr$ acts by $r_0.$
\end{enumerate}

\end{theorem}

\subsection{} We recall next the classification and construction of
irreducible tempered $\bH$-modules following \cite{L3}. Suppose that $r_0\in \bR_{>0}.$

\begin{definition}Denote $V^{*,+}=\{\xi\in V^*: \xi(\alpha^\vee)>0,\text{ for all }\al\in \Pi\}$ the set of dominant elements of $V^*.$ Notice that $\delta(V^{*,+})=V^{*,+}.$ An irreducible $\bH$-module $X$ is called tempered if every $S(V_\bC^*)$ weight $\nu\in V_\bC$ of $X$ satisfies
$$\xi(\Re\nu)\le 0,\ \text{ for all }\xi\in V^{*,+}.$$
If all the inequalities are strict, then $X$ is called a discrete series module.
\end{definition}

\begin{theorem}[{\cite[Theorem 1.21]{L3}}]\label{Lusztig-class} Let
  $e,s,\psi$ be as in Theorem \ref{Lusztig-irred} and $r_0\in
  \bR_{>0}$. 
\begin{enumerate}
\item The simple module
  $\overline E_{e,s,r_0,\psi}$ is tempered if and only if there exists
  a Lie triple $\{e,h,f\}$ such that $[s,h]=0$, $[s,f]=-2r_0f$, and
  $\operatorname{ad}(s-r_0 h):\fg\to\fg$ has no real eigenvalues.

Moreover, in this case $\overline E_{e,s,r_0,\psi}=E_{e,s,r_0,\psi}.$

\item The module $\overline E_{e,s,r_0,\psi}(=E_{e,s,r_0,\psi})$ is a
  discrete series module if and only if $e$ is distinguished and $s=r_0h.$
\end{enumerate}
\end{theorem}

\subsection{}\label{sec:tempered} In light of Theorem \ref{Lusztig-class}, fix a Lie
triple $\phi=\{e,h,f\}$, and consider the
module $E_{e,h,1,\psi}$, with $\psi\in \widehat{A(e,h)}_0.$ This is a
simple tempered $\bH$-module (with real central character) on which
$\bfr$ acts by $r_0=1.$ So we assume from now on that $r_0=1$ and drop
it from the notations.

Assume that $\d \neq \id$. We extend the $\bH$-module structure on $\overline E_{ e, h, \psi}$ to a $\bH_\#$-module structure. The construction is similar to section \ref{w3}. 
Assume, as we may, that $h\in \fk t^\delta$. By \cite[Lemma 3.7.3 and Remark 3.7.5 (ii)]{CM}, $A(e)=A(h, e)$. Then the natural map $$\{g \in G; \d(\phi)=\Ad(g)(\phi)\} \to \{g \in G; \d(e)=\Ad(g)(e)\}$$ induces a bijection on the connected components. Let $g \in G$ with $\d(\phi)=\Ad(g)(\phi)$ and that the image of $g$ in $\{g \in G; \d(e)=\Ad(g)(e)\}$ is in the same connected component as the element $g_1$ in section \ref{nice-choice}. 
Then $\Ad(g)^* \circ \d^*: H_*(\C B^h_e) \to H_*(\C B^h_e)$. Since $\d(h)=h$, $\Ad(g)^* \circ \d^*: E_{e,h} \to E_{e,h}$. Similar to \ref{e:5.2.1}, we have
  
\begin{equation}\label{e:8.5.1}
(\Ad(g)^* \circ \d^*)^2=\id.
\end{equation}

We have the following commuting diagram 
\begin{equation}\label{e:8.5.2}
\xymatrix{\bH \times E_{e,h} \ar[rr]^-{(\d, \d^*)} \ar[d] & & \bH \times E_{\d(e),h} \ar[rr]^-{(\id, \Ad(g)^*)} \ar[d] & & \bH \times E_{e,h} \ar[d] \\ E_{e,h} \ar[rr]^-{\d^*} & & E_{\d(e),h} \ar[rr]^-{\Ad(g)^*} & & E_{ e,h}.}
\end{equation}

Define the action of $\d \in \bH_\#$ on $E_{e,h}$ by $\Ad(g)^* \circ
\d^*$. By (\ref{e:8.5.1}) and (\ref{e:8.5.2}), this gives an action of
$\bH_\#$ on $E_{e,h}$. The map $\Ad(g)^* \circ \d^*: H_*(\C B^h_e) \to H_*(\C B^h_e)$ induces an action on $\widehat{A(e, h)}_0$, which we denote by $\d$. 

By \cite[Lemma 3.7.3 and Remark 3.7.5 (ii)]{CM}, $A(e)=A(h, e)$. By \cite[10.13]{L2},  there is an isomorphism of $W$-modules:
\begin{equation}\label{temp-restr}
E_{e,h, \psi} \cong X_1(e,\psi),
\end{equation} 
where $X_1(e,\psi)=H^*(\C B_e)^\psi\otimes\sgn$ is the
Springer $W$-representation from section \ref{sec:2}.
Since the $W$-structure of the module does not change under the $\delta$-twist, (\ref{temp-restr}) implies that $\delta(\psi)=\psi.$ 

Hence for any $\psi \in \widehat{A(e)}_0$, $E_{e, h,\psi}$ is an
$\bH_\#$-module such that (\ref{temp-restr}) is an isomorphism of
$W_\#$-modules, where the $\delta$-action on $X_1(u,\psi)$ is as in
section \ref{w3}.

\subsection{} Define
the Clifford algebra $C(V)$ of $(V,(~,~))$
to be the real associative algebra with identity generated by
$\{\xi\in V\}$ subject to the relations
\begin{equation}
\xi\cdot\xi'+\xi'\cdot\xi=-2(\xi,\xi'),\quad \xi,\xi'\in V.
\end{equation}
The algebra $C(V)$ is naturally $\bZ_{\ge 0}$-filtered, where the
$n$-th space $C_n(V)$ in the filtration is the span of all elements of $C(V)$
which are products of at most $n$ elements of $V$. The associated
graded algebra is $\wedge V.$

The Clifford algebra $C(V)$ is also $\bZ/2\bZ$-graded
$C(V)=C(V)_\ev+C(V)_\odd$ by the parity of the degree of
homogeneous elements in the filtration just defined. Let $\ep:C(V)\to
C(V)$ be the involution which is $+1$ on $C(V)_\ev$ and $-1$ on $C(V)_\odd.$

Let $^t: C(V)\to C(V)$ be the anti-automorphism defined by
$\xi^t=-\xi$ for all $\xi\in V.$

Define the pin group
\begin{equation}
\Pin(V)=\{g\in C(V)^\times|~ g^t=g^{-1}\text{ and }\ep(g)\cdot\xi\cdot g^{-1}\in V,\text{ for all
}\xi\in V\}.
\end{equation}
This is a central double extension of $O(V)$ with the projection map
$p:C(V)\to O(V),$ $p(g)(\xi)=\ep(g)\cdot\xi\cdot g^{-1},$
$g\in\Pin(V),$ $\xi\in V.$
Since $W\subset O(V),$ one considers
\begin{equation}
\wti W=p^{-1}(W)\subset \Pin(V),
\end{equation}
a central double extension of $W$.

When $\dim V$ is even, $C(V)$ is a central simple algebra, and
therefore it has a unique complex simple module $S$ of dimension $2^{\dim
  V/2}.$ When $\dim V$ is odd, $Z(C(V))$ is two dimensional and
$C(V)=C(V)_\ev\otimes Z(C(V))$. In this case, the unique simple module
of $C(V)_\ev$ can be extended in two inequivalent ways to $C(V)$,
$S^+$ and $S^-$. In what follows, we will refer to any one of $S$,
$S^+,$ $S^-$ as a spin module of $C(V)$. For convenience, we also set
\begin{equation}\label{spin-sum}
\C S=\begin{cases}S,&\dim V\text{ even},\\S^++S^-,&\dim V\text{ odd}.\end{cases}
\end{equation}
One can restrict every spin module to $\Pin(V)$ and furthermore to
$\wti W.$ Since we assumed $V^W=0$ (semisimple), $\wti W$ generates
$C(V),$ and therefore every spin module is an irreducible $\wti
W$-representation. 

We have
\begin{equation}\label{spin-squared}
\C S\otimes \C S=a_V \wedge V
\end{equation}
as $W$-representations (or $O(V)$-representations), where $a_V=1$ if $\dim V$ is even, and $a_V=2$ if $\dim V$ is odd.

The group $\wti W$ admits a Coxeter-like presentation. Denote by $m(\al,\beta)$ the order in $W$ of $s_\al s_\beta.$ Then:
\begin{equation}\label{Cox-Wtilde}
\wti W=\langle -1,\wti s_\al,\al\in\Pi~|~ (-1)^2=1,\ (\wti s_\al \wti s_\beta)^{m(\al,\beta)}=-1\rangle. 
\end{equation}

\subsection{} Let $-1$ be the automorphism of $O(V)$ induced by
$\xi\mapsto -\xi$ on $V$ and recall $W_\#=\langle W,\delta\rangle=\langle W,-1\rangle\subset O(V).$ Define $\wti W_\#=p^{-1}(W_\#)\subset \Pin(V).$

We fix an orthonormal basis $\{\xi_1,\xi_2,\dots,\xi_n\}$ of $V$ permuted by $\delta$, and set 
\begin{equation}\label{z-element}
z=\xi_1\x_2\dots\xi_n\in \Pin(V).
\end{equation}

\begin{lemma} The element $z$ satisfies the following properties:
\begin{enumerate}
\item $z^2=(-1)^{\frac {n(n+1)}2}$;
\item $p(z)=-1\in O(V)$;
\item $z \xi=(-1)^{n-1} \xi z$, $\xi\in V$.
\end{enumerate}
\end{lemma}

\begin{proof}
Straightforward.
\end{proof}

Therefore, $z$ is a central element in $C(V)$ if $\dim V$ is odd, and $z$ is a central element of $C(V)_\ev$, when $\dim V$ is even. When $\dim V$ is even, denote by $S^\pm$ the two constituents in the restriction of spin module $S$ to $C(V)_\ev.$ With this notation, $z$ acts by scalars on $S^\pm$ in both cases $\dim V$ odd or even. Since the trace of $-1\in O(V)$ on $\wedge V$ is zero, (\ref{spin-squared}) implies that the trace of $z$ in $S^++S^-$ is zero as well. Therefore, the scalars by which $z$ acts on $S^+,S^-$ differ by a sign, i.e.,
\begin{equation}
z|_{S^+}=-z|_{S^-}=c\in \bC.
\end{equation}
(We have $c^2=(-1)^{n(n+1)/2}$, but we will not need to use this fact.)

We will use the formula (\cite[Lemma 3.4]{BCT})
\begin{equation}
\wti w\cdot \xi\cdot \wti w^{-1}=\sgn(\wti w) w(\xi),\ \wti w\in \wti W,\ \xi\in V.
\end{equation}

Fix once for all $\wti w_0\in p^{-1}(w_0)$ and set 
\begin{equation*}
\wti\delta=\wti w_0z,\text{ so that }p(\wti \delta)=\delta.
\end{equation*}
It is easy to check that
\begin{equation}\label{delta-commute}
\wti\delta\xi=(-1)^{n+\ell(w_0)}\delta(\xi)\wti\delta, \ \xi\in V.
\end{equation}
The group $\wti W_\#$ is generated by $\wti W$ and $z$. It is also
generated by $\wti W$ and $\wti \d$. Set
\begin{equation}
\wti W'=\begin{cases} \wti W \cap C(V)_\ev,&\dim V\text{ even},\\\wti W,&\dim V\text{ odd},\end{cases}
\end{equation}
and
\begin{equation}
\wti W'_\#=\begin{cases} \langle \wti W \cap C(V)_\ev, z \rangle,&\dim V\text{ even},\\\wti W_\#,&\dim V\text{ odd}.\end{cases}
\end{equation}
The discussion above shows that $S^\pm$ are $\wti W'_\#$-representations.

\subsection{}\label{sec:twisted-Dirac} Following \cite[Definition
3.1]{BCT}, define the Dirac element
\begin{equation}
\C D=\sum_{i=1}^n \wti \xi_i\otimes\xi_i\in \bH\otimes C(V)\subset \bH_\#\otimes C(V).
\end{equation}
Let $\rho: \bC[\wti W_\#]\to \bH_\#\otimes C(V)$ be the linear map extending $\wti w\mapsto p(\wti w)\otimes \wti w,$ $\wti w\in \wti W$ and $\wti\delta\mapsto \delta\otimes \wti\delta$. Notice that $\rho(z)=w_0\delta\otimes z\in \bH\delta\otimes C(V).$

The algebra $C(V)$ also has a $*$-operation defined by the transpose
map $^t$. On $\Pin(V)$ this corresponds to the inversion
operation. With respect to this operation, the spin modules $S,S^\pm$
admit positive definite invariant forms (i.e., they are unitary).  

Define the Casimir element of $\bH$ (\cite[Definition 2.3]{BCT}):
\begin{equation}
\Omega=\sum_{i=1}^n \xi_i^2.
\end{equation}
This is an element in $S(V)^{W_\#}$, thus central in $\bH_\#$. The central characters of irreducible $\bH_\#$-modules are parameterized by $W_\#$-orbits in $V_\bC$.   By
\cite[Lemma 2.5]{BCT}, if $(\pi,X)$ is an irreducible $\bH_\#$-module
with central character $W_\#\cdot \nu$, then $\pi(\Omega)$ acts on $X$
by the scalar $(\nu,\nu).$

Let $\Phi^\vee\subset V$ be corresponding coroots and $\Phi^{\vee,+}$
the positive coroots, $\Phi^{\vee,-}$, the negative coroots. Define
the Casimir element of $\wti W$ (\cite[section 3.4]{BCT})
\begin{equation}\label{CasimirW}
\Omega_{\wti W}=(-1)\sum_{\al,\beta\in\Phi^+, s_\al(\beta)\in \Phi^-}
|\al^\vee||\beta^\vee|\wti s_\al \wti s_\beta\in \bC[\wti W']^{\wti W},
\end{equation}
where $(-1)$, $\wti s_\al$  are as in (\ref{Cox-Wtilde}).

\begin{proposition}\label{p:twistedDirac} The element $\C D$ has the following properties:
\begin{enumerate}
\item $\C D^*=\C D$.
\item $\rho(\wti w)\C D=\sgn(\wti w)\C D \rho(\wti w),$ for all $\wti w\in\wti W$;
\item $\rho(z)\C D=(-1)^{\ell(w_0)+n} \sgn(w_0)\C D\rho(z)$.
\item $\C D^2=-\Omega\otimes 1+\bfr^2\rho(\Omega_{\wti W})\in \bH\otimes C(V)$.

\end{enumerate}

\end{proposition}

\begin{proof}

(1) Straightforward.

(2) This is \cite[Lemma 3.4]{BCT}.

(3) We have $z=\wti w_0^{-1}\wti\delta.$ From  (\ref{delta-commute}),
we see that $\rho(\wti\delta)\C D=(-1)^{\ell(w_0)+n} \C
D\rho(\wti\delta)$. The claim now follows from (2).

(4) This is \cite[Theorem 3.5]{BCT}.

\end{proof}

\begin{definition} If $X$ is an $\bH_\#$-module, and $S$ is a spin
  $C(V)$-module (when $\dim V$ is odd, there are two choices), left action by $\C D$ defines the Dirac operator (of $X$ and $S$) $$D_\#: X\otimes S\to X\otimes S.$$ Define the extended Dirac cohomology of $X$ (with respect to $S$)
\begin{equation} 
H^D_\#(X)=\ker (D_\#)/(\ker D_\#)\cap (\im D_\#).
\end{equation}
By Proposition \ref{p:twistedDirac}(2),(3), $H^D_\#(X)$ is a $\wti
W_\#$-representation. From Proposition \ref{p:twistedDirac}(1), when
$X$ is a $*$-unitary $\bH_\#$-module, $H^D_\#(X)=\ker D_\#$.
\end{definition}

\begin{definition}\label{d:Dirac-cohom}
Suppose $\dim V$ is even. By restriction, $D_\#$ defines two operators 
\begin{equation}
D_\#^\pm: X\otimes S^\pm\to X\otimes S^\mp. 
\end{equation}
Suppose $\dim V$ is odd. In this case, $S^+$ and $S^-$ are realized on
the same vector space $U$ (coming from the unique simple module of
$C(V)_\ev$). Then, as in \cite[section 2.9]{CT}, $D_\#: X\otimes S^+\to X\otimes S^+$ can
be composed with the vector space identity map $S^+\to S^-$ to yield $D_\#^+:
X\otimes S^+\to X\otimes S^-$. Similarly, define $D_\#^-$. 

In both cases, set
\begin{equation}
H^{D^\pm}_\#(X)=\ker (D^\pm_\#)/(\ker D^\pm_\#)\cap (\im D^\mp_\#),
\text{ and }  I_\#(X)=H^{D^+}_\#(X)-H^{D^-}_\#(X).
\end{equation}
We call $I_\#(X)$ the extended Dirac index. By Proposition
\ref{p:twistedDirac}(2),(3), $H^{D^\pm}_\#(X)$ are $\wti
W'_\#$-representations, and $I_\#(X)$ is a virtual $\wti W'_\#$-module.

Notice that 
\begin{equation}
D^\pm_\#\circ D^\mp_\#=-\Omega\otimes 1+\bfr^2\rho(\Omega_{\wti W}),
\end{equation}
by Proposition 6.6(4), since $D^\pm_\#$ are given by the left action of
$\C D.$
\end{definition}

\begin{proposition}\label{p:twisted-Dirac}
 For every $\bH_\#$-module $X$, we have $I_\#(X)=X\otimes (S^+-S^-)$
 as virtual $\wti W'_\#$-modules. In particular, for $\wti w\in \wti W'$,
\begin{equation}
\begin{aligned}
\tr(\wti w, I_\#(X))&=\tr(w,X)\tr(\wti w,S^+-S^-),\\
\tr(\wti w z,I_\#(X))&=c\tr(ww_0\delta,X) \tr(\wti w,\C S),
\end{aligned}
\end{equation}
where $\C S$ is as in (\ref{spin-sum}), and $c$ is the scalar by which $z$ acts in $S^+$.
\end{proposition}

\begin{proof}
The first claim is proved identically with \cite[Lemma 4.1]{COT}. The
second claim is immediate from the first since $z$ acts by $c$ in
$S^+$ and by $-c$ in $S^-$, and $\rho(z)=w_0\delta\otimes z.$ 
\end{proof}

\subsection{} We are now in position to compute the extended Dirac index of the
simple tempered $\bH_\#$-modules $E_{e,h,\psi}.$

\begin{theorem}\label{t:delta-twisted} Let $E_{e,h,\psi}$ be a simple
  tempered $\bH_\#$-module as above. 

\begin{enumerate}
\item The $\delta$-twisted trace of $E_{e,h,\psi}$ on $W$ is given by:
\begin{equation}\label{conj-twisted}
\tr(ww_0\delta,E_{e,h,\psi})=(-1)^{d_e}\sgn(w_0) X_{-1}(e,\psi)(w),\quad \text{for all }w\in W.
\end{equation}

\item The extended Dirac index of $E_{e,h,\psi}$ is
given by the formula:
\begin{equation}\label{twisted-index}
\begin{aligned}
\tr(\wti w,I_\#(E_{e,h,\psi}))&=\tr (\wti w,X_1(e,\psi)\otimes (S^+-S^-)),\\
\tr(\wti wz,I_\#(E_{e,h,\psi}))&=c'\tr(\wti w, X_{-1}(e,\psi)\otimes
\C S),
\end{aligned}
\end{equation}
with $\wti w\in \wti W',$ $z$ is as in (\ref{z-element}), $c'=(-1)^{d_e}\sgn(w_0) c$, where $c$ is the scalar from Proposition \ref{p:twisted-Dirac}. 

\end{enumerate}
\end{theorem}

\begin{proof} The above construction of the $\delta$-action on
  $E_{e,h,\psi}$ gives
  $\tr(w\delta,E_{e,h,\psi})=X_1(u,\psi)(w\delta).$ Then (1) follows
  from Corollary \ref{c:action-1}.

For (2), Proposition \ref{p:twisted-Dirac} implies that $\tr(\wti
 w,I_\#(E_{e,h,\psi}))=tr(w,E_{e,h,\psi})\tr(\wti w,S^+-S^-)$ and 
 $\tr(\wti
 wz,I_\#(E_{e,h,\psi}))=c\tr(ww_0\delta,E_{e,h,\psi})\tr(\wti w,\C
 S).$ The formula now follows from part (1) and (\ref{temp-restr}).

\end{proof}

\subsection{}\label{sec:7.9} The analogue of Vogan's conjecture from real reductive groups in the setting of the graded affine Hecke algebra was stated and proved in \cite{BCT}. We need the following algebraic form.

\begin{theorem}[{\cite[Theorem 4.2]{BCT}}]\label{t:Vogan}
For every $y\in Z(\bH)$, there exist a unique element $\zeta(y)\in Z(\bC[\wti W])$ and an element $a\in \bH\otimes C(V)$ such that 
\begin{equation}\label{zeta}
y\otimes 1=\rho(\zeta(y))+\C Da+a\C D,
\end{equation}
as elements of $\bH\otimes C(V).$ Moreover, the map $\zeta: Z(\bH)\to Z(\bC[\wti W])$ is an algebra homomorphism. 
\end{theorem}
In fact, as one can see from the proof of \cite[Theorem 4.2]{BCT}
(cf. \cite[Theorem 3.2]{COT}),  the element $a$ belongs to $\bH\otimes C(V)_\odd$. It is also noticed in \cite[Corollary
3.3]{COT} that the image of the map
$\zeta$ lies in $\bC[\wti W']^{\wti W}.$
 Following \cite[Definition
4.5]{COT} (compare also with \cite[Definition 4.3]{BCT}), if
$\wti\sigma$ is an irreducible $\wti W'$-representation, one can
attach canonically a homomorphism $\chi^{\wti\sigma}: Z(\bH)\to\bC$
(i.e., a central character of $\bH$-modules) by the
requirement $$\chi^{\wti\sigma}(y)=\wti\sigma(\zeta(y)), \text{ for
  all }y\in Z(\bH).$$
Notice that if $\wti\sigma$ is an irreducible $\wti W$-representation,
the central character $\chi^{\wti\sigma_1}$ is the same for all
irreducible $\wti W'$-representations that appear in the restriction
of $\wti\sigma$ to $\wti W'$. Thus we can also denote
$\chi^{\wti\sigma}$ for an irreducible $\wti W$-representation
$\wti\sigma.$

 Let $W\cdot\nu_{\wti\sigma}$ denote the $W$-orbit in $V_\bC$ corresponding to $\chi^{\wti\sigma}.$

\smallskip

We following is a slight sharpenning of \cite[Theorem 4.4]{BCT}. 

\begin{corollary}\label{c:vogan} Let $\ep\in\{+,-\}.$ Suppose $(\pi,X)$ is an irreducible $\bH_\#$-module with central character $W_\#\cdot \nu$ and that $\wti\sigma$ is an irreducible $\wti W'$-representation such that 
$$\Hom_{\wti W'}[\wti\sigma, H^{D^\ep}_\#(X)]\neq 0,$$
where $H^{D^\ep}_\#(X)$ is as in Definition \ref{d:Dirac-cohom}.
Then $W_\#\cdot\nu=W_\#\cdot \nu_{\wti\sigma}.$
\end{corollary}

\begin{proof}
We show how the claim follows from Theorem \ref{t:Vogan}. This is
analogous with the proof of \cite[Theorem 4.4]{BCT}, but we need some
minor modification because
we are considering $D^\pm$ (rather than operators $D$).
 
When $\dim V$ is odd, $S^\ep$ and $S^{-\ep}$ are realized on the same
vector space $U^\ep=U^{-\ep}$. When $\dim V$ is even, $S^\ep$ and $S^{-\ep}$ are realized on the
vector spaces $U^\ep$ and $U^{-\ep}$, respectively.

Let $\gamma$ denote Clifford multiplication by elements in $C(V)$ on the spin module
$S^\ep$. When $\dim V$ is
odd, $\gamma$ is a $C(V)$-action on $U^\ep.$ When $\dim V$ is even,
$\gamma(\xi),$ where $\xi\in C(V)_\odd$ take $U^\ep$ to $U^{-\ep}.$

Let $y\in Z(\bH_\#)$ and $\wti x\in X\otimes U^\ep$ be an element in the $\wti W'$-isotypic component of $\wti\sigma$ in $H^{D^\ep}_\#(X).$ Then $(\pi(y)\otimes\gamma(1)) \wti x=\chi_\nu(y)\wti x$ and $(\pi\otimes\gamma^\ep)(\rho(\zeta(y)))\wti x=\wti\sigma(\rho(\zeta(y)) \wti x=\chi^{\wti\sigma}(y)\wti x.$ Together with (\ref{zeta}), it follows that:
\begin{align*}
(\chi_\nu(y)-\chi^{\wti\sigma}(y))\wti x&=(\pi\otimes\gamma)(y\otimes 1-\rho(\zeta(y)))\wti x\\
&=(\pi\otimes\gamma)(\C D a +a\C D) \wti x, \\
&=(\pi\otimes\gamma)(\C D a  \wti{x}).
\end{align*}
By the discussion above $a\in \bH\otimes
C(V)_\odd$, and so when $\dim V$ is
even, $a x\in X\otimes U^{-\ep}.$ Therefore,
regardless of the parity of $\dim V$, the right hand side is in $\im
D_\#^{-\ep}$, and it follows by  the definition of $H^{D^\ep}_\#(X)$
that it must be zero. In conclusion, $\chi_\nu=\chi^{\wti\sigma}.$
\end{proof}

\section{Spin Weyl group representations}\label{sec:5}

\subsection{} 
We denote by $\widehat{\wti W}_\gen$ the set of (isomorphism classes of) irreducible genuine representations of $\wti W$, i.e.,  the irreducible representations of $\wti
W$ which do not factor through $W.$
Let $R(\wti W)_\gen$ be the subspace of $R(\wti W)$ spanned by $\widehat{\wti W}_\gen$. 
Denote $$\Sg: R(\wti W)\to R(\wti W), \qquad \Sg(\sigma)=\sigma\otimes\sgn.$$
Let $R(\wti W)^\Sg$ denote the $(+1)$-eigenspace of $\Sg.$ Notice that $\C S \otimes \sgn=\C S$. 
Define the linear map
\begin{equation}\label{d:iota}
\iota: R(W)\to R(\wti W)_\gen^\Sg,\quad \iota(\sigma)=\sigma\otimes \C S.
\end{equation}

\begin{proposition}\label{p:iota}
 The map $\iota$ induces an injective linear map
  $\iota: \overline R_{-1}(W)\to R(\wti W)_{\gen}^\Sg$ such that for $\sigma,\sigma'\in
\overline R(W)$, 
$$\langle\iota(\sigma),\iota(\sigma')\rangle_{\wti
  W}=a_V\langle\sigma,\sigma'\rangle^{-1}_W.$$
Moreover, $\iota(X_{-1}(e, \phi))\neq 0$ if and only if $e\in
\C N^\sol$ and
\begin{equation}\label{e:X-rel}
\langle\iota(X_{-1}(e, \phi)),\iota(X_{-1}(e',\phi'))\rangle_{\wti
  W}=\begin{cases}a_V\langle\phi,\phi'\rangle^{-1}_{A(e)},&\text{ if
  }e=e'\in \C N^\sol\\
0, &\text{ otherwise}.
\end{cases}
\end{equation}
\end{proposition}

\begin{proof}
Since $\C S^*\cong \C S,$ the first claim is immediate from (\ref{spin-squared}) and the definition of $\langle~,~\rangle^{-1}_{W}.$ The second claim follows from Proposition \ref{p:6.1}(2).

By Theorem \ref{t:2.7} in the case $q=-1$, the map $\phi\to X(e, \phi)$
induces an isometric isomorphism $\overline R_{-1}(A(e))_0\to
R_{-1}^e(W).$ Composing with $\iota$, this implies that
$\iota(X_{-1}(e, \phi))\in R(\wti W)_\gen$, $e \in \C N^\sol$, is
nonzero and that (\ref{e:X-rel}) holds.
\end{proof}

We relate these facts with the
extended Dirac index from section \ref{sec:twistedirac}.

\begin{lemma}\label{l:support}
$W_{\text{(-1)-ell}}\subset \ker\sgn$.
\end{lemma}
\begin{proof}
Suppose $w\in W$ is such that $\det_{V}(1+w)\neq 0.$ This means that
$w$ acting on $V$ does not have the eigenvalue $-1$. Since $V$ is a real
representation, the only real eigenvalue of $w$ is $1$ and the complex
eigenvalues come in pairs. Thus $\sgn(w)=\det_V(w)=1.$
\end{proof}

\begin{proposition}\label{p:tempered-index} Let $E_{e,h,\phi}$ be a simple tempered
  $\bH_\#$-module as in section \ref{sec:tempered}. The extended Dirac
  index $I_\#(E_{e,h,\phi})\neq 0$ if and only if $e\in \C N^\sol.$
\end{proposition}

\begin{proof}
In one direction, suppose $e\in \C N^\sol.$ By Proposition
\ref{p:iota}, $\iota(X_{-1}(e,\phi))\neq 0$. By (\ref{spin-squared}),
$\C S$ is supported on $W_{\text{(-1)-ell}}$ and Lemma
\ref{l:support} implies that the support of $\iota(X_{-1}(e,\phi))$ is
in $\wti W'.$ Therefore, (\ref{twisted-index}) says that the
restriction of 
$I_\#(E_{e,h,\phi})$ to $\wti W'z$ is nonzero.

For the converse, suppose that $I_\#(E_{e,h,\phi})\neq 0.$ Again by
(\ref{twisted-index}), there are two cases. If there exists $\wti w\in
\wti W'$ such that $\tr(\wti w,I_\#(E_{e,h,\phi}))\neq 0,$ then
$\tr(\wti w,X_1(e,\phi)\otimes (S^+-S^-))\neq 0.$ By \cite[Proposition
3.1]{CT}, which is the   analogue of Proposition \ref{p:iota}
above, it follows that $$\langle X_1(e,\phi),X_1(e,\phi)\rangle^1_{A(u)}\neq 0.$$
Then Proposition \ref{p:6.1}(1) implies that $e$ is
quasidistinguished and so $e\in \C N^\sol.$

If, on the other case, $\tr(\wti wz,I_\#(E_{e,h,\phi}))\neq 0,$ then
$\tr(\wti w,\iota(X_{-1}(e,\phi))\neq 0,$ and by Proposition
\ref{p:iota}, $e\in \C N^\sol.$
\end{proof}

\subsection{} For $e \in \C N^\sol$ and $\phi\in\widehat {A(e)}_0$,
let $[\phi]$ be the image of $\phi$ in $\overline
R_{-1}(A(e))_0$. Denote
$m_{e,[\phi]}(\wti\sigma)=\langle\wti\sigma,\iota(X_{-1}(e, \phi))\rangle_{\wti
W}$, for every $\wti\sigma\in\widehat{\wti W}_\gen$, and set
\begin{equation}
\Irr_{e,[\phi]}\wti W=\{\wti\sigma\in \widehat{\wti W}_\gen:
m_{e,[\phi]}(\wti\sigma)\neq 0\},\quad \Irr_e\wti W=\cup_{\phi}
\Irr_{e,[\phi]}\wti W.
\end{equation}
It is clear that
$\Irr_{e,[\phi]}\wti W$ is closed under tensoring with $\sgn$, in fact
$m_{e,[\phi]}(\wti\sigma)=m_{e,[\phi]}(\wti\sigma\otimes\sgn)$.

Proposition \ref{p:tempered-index} and Corollary \ref{c:vogan} yield
the following remarkable fact. 

\begin{corollary}\label{Casimir-action} Let $\{e,h,f\}$ be a Lie triple in $\fg$, $h\in V^\delta$,
  such that $e \in \C N^\sol$, and 
$\wti\sigma\in \Irr_{e}\wti W$. Then $W\cdot
\nu_{\wti\sigma}=W\cdot h,$ where $\nu_{\wti\sigma}$ is the central
character of $\bH$ defined by $\wti\sigma$ in section \ref{sec:7.9}. 

In particular, $\wti\sigma(\Omega_{\wti
  W})=(h,h),$ where $\Omega_{\wti W}$ is as in (\ref{CasimirW}).
\end{corollary}

\begin{proof}
Let $\wti\sigma\in \Irr_e\wti W.$ There exists
$\phi\in\widehat{A(e)}_0$ such that $\Hom_{\wti
  W}[\wti\sigma,\iota(X_{-1}(e,\phi))]\neq 0.$ Let $\wti\sigma_1$ be
an irreducible $\wti W'$-representation such that $\Hom_{\wti
  W'}[\wti\sigma_1,\wti\sigma]\neq 0$ and
$\langle\wti\sigma_1,\iota(X_{-1}(e,\phi))\rangle_{\wti W'}\neq 0.$

Write $I_\#(E_{e,h,\phi})=\sum_j a^j\wti\sigma_\#^j,$ for
$\wti\sigma_\#^j$ irreducible distinct $\wti W_\#'$-representations, and $a^j\in
\bZ^*$. Suppose further that $\wti\sigma_\#^j=\sum_i \wti\sigma_i^j$ as
$\wti W'$-representations, where $\wti\sigma_i^j$ are irreducible
$\wti W'$-representations. Since $z$ commutes with $\wti W'$, it acts
by a nonzero scalar $u^j_i$ (in fact, a fourth root of $1$) on
$\wti\sigma^i_j.$ Thus
$$\tr(\wti w z,I_\#(E_{e,h,\phi}))=\sum_j\sum_i a^j u^j_i \tr(\wti w,\wti\sigma^j_i).$$
On the other hand, 
by
Proposition \ref{p:tempered-index},
the left hand side equals (up to a nonzero scalar) $\tr(\wti w,
\iota(X_{-1}(e,\phi))$, and so $\tr(\wti w,\wti\sigma_1)$ appears in
the linear combination. By
the linear independence of irreducible $\wti W'$-characters, it
follows that there exist $i,j$ such that
$\wti\sigma^j_i=\wti\sigma_1.$ In other words, there exists $j$ such
that $\wti\sigma_\#^j$ contains $\wti\sigma_1.$ 

Since $\wti\sigma_\#^j$ occurs in $I_\#(E_{e,h,\phi}),$ it must occur
in one of the spaces  
$H^{D^\pm}_\#(E_{e,h,\phi})$ for a choice of sign $\pm.$ This implies
that $\Hom_{\wti W'}[\wti\sigma_1,H^{D^\pm}_\#(E_{e,h,\phi})]\neq 0.$
Corollary \ref{c:vogan} says that
$W_\#\cdot h=W_\#\cdot \nu_{\wti\sigma_1}=W_\#\cdot
\nu_{\wti\sigma}$. 

Since $h$ is
$\delta$-stable, the first claim of the corollary is proved. For the second claim, it is
sufficient to notice that by definition
$(\nu_{\wti\sigma},\nu_{\wti\sigma})=\wti\sigma(\Omega_{\wti W}).$
\end{proof}

\subsection{} We state the main results of this section.

\begin{theorem}\label{l:exhaustion}
$\widehat{\wti W}_\gen=\sqcup_{e \in G \backslash \C N^\sol} \Irr_e\wti W.$
\end{theorem}

\begin{proof}
Let $\wti\sigma$ be an irreducible genuine $\wti
W$-representation. Then $\wti\sigma\otimes\C S$ is a
$W$-representation. Since $\{X_{-1}(e, \phi): e\in G\backslash
\C N,\phi\in\widehat{A(e)}_0\}$ is a basis of $R(W)$, there exists
$(e, \phi)$ such that $\langle \wti\sigma\otimes\C
S,X_{-1}(e, \phi)\rangle_W\neq 0$, or equivalently $\langle
\wti\sigma,\iota(X_{-1}(e, \phi))\rangle_{\wti W}\neq 0.$ In
particular, $\iota(X_{-1}(e, \phi))\neq 0,$ thus by Proposition \ref{p:iota}, $e \in \C N^\sol$. The disjointness follows from Corollary \ref{Casimir-action}. 
\end{proof}

\begin{theorem}\label{t:main-spin} Let $e \in \C N^\sol$ and $\phi\in\widehat {A(e)}_0$ be given.
\begin{enumerate}
\item For every $\wti\sigma\in \Irr_{e,[\phi]}\wti W,$ $m_{e,\phi}(\wti\sigma)=\langle\sigma(e, \phi)\otimes\C
  S,\wti\sigma\rangle_{\wti W}$, and in particular,
$\iota(X_{-1}(e, \phi))=X_{-1}(e, \phi)\otimes\C S$ 
is the character of a genuine representation of $\wti W.$
\item If 
$\langle\phi,\phi\rangle^{-1}_{A(e)}=1,
$
in particular, when $e$ is distinguished,
then 
$\Irr_{e,[\phi]}\wti W=\{\wti\sigma(e,[\phi])\}$, where
$\wti\sigma(e,[\phi])$ is irreducible $\sgn$ self dual if $\dim V$ is
even, and $\Irr_{e,[\phi]}\wti W=\{\wti\sigma(e,[\phi])^+,\wti\sigma(e,[\phi])^-\}$, where
$\wti\sigma(e,[\phi])^\pm$ are irreducible $\sgn$ dual representations
when $\dim V$ is odd.
\item If $\langle \phi,\phi'\rangle^{-1}_{A(e)}=0$, $\phi\neq \phi'$,
  in particular, when $e$ is distinguished, then $\Irr_{e,[\phi]}\wti
  W\cap \Irr_{e,[\phi']}\wti W=\emptyset.$
\end{enumerate}
\end{theorem}

\begin{proof}
Since the matrix of Green polynomials $K(q)$ is upper triangular with
$1$ on the diagonal, $K(q)^{-1}$ is again upper triangular with $1$ on
the diagonal and polynomials in $q$ above the diagonal. It is
well-known that the $W$-types in $X_q(e,\phi)$, other than $\sigma(e, \phi)$ are all of the form $\sigma(e',\phi')$ with $e'>e.$  Thus, in $R_q(W)$, we have:
\begin{equation}\label{e:upper-triang}
X_q(e,\phi)=\sigma(e, \phi)-\sum_{e'>e} K(q)^{-1}_{(e, \phi),(e',\phi')} X_q(e',\phi').
\end{equation}
Apply this identity when $q=-1$ and tensor with $\C S$:
\begin{equation}\label{e:upper-triang-1}
\iota(X_{-1}(e, \phi))=\sigma(e, \phi)\otimes\C S-\sum_{e<e', e'\in \C N^\sol}K(-1)^{-1}_{(e, \phi),(e',\phi')} \iota(X_{-1}(e', \phi')),
\end{equation}
an identity in $R(\wti W)_\gen.$ Suppose that $\wti\sigma$ occurs in
the LHS. By Corollary \ref{Casimir-action},
$W\cdot\nu_{\wti\sigma}=W\cdot h$, where $h$ is a neutral element for a
Lie triple of $e$. In particular, $\wti\sigma$ cannot
occur in any $\iota(X_{-1}(e',\phi'))$ with $G\cdot e'\neq G\cdot e$,
since $h$ determines $G\cdot e.$ Thus $\wti\sigma$ can only occur in
$\sigma(e, \phi)\otimes \C S.$ This proves (1).


For (2), suppose $\langle\phi,\phi\rangle^{-1}_{A(e)}=1$. (The fact
that this is always the case when $u$ is distinguished is immediate.)
By Proposition \ref{p:iota},
$\langle\iota(X_{-1}(e, \phi)),\iota(X_{-1}(e, \phi))\rangle_{\wti
  W}=a_V$. Since $\iota(X_{-1}(e, \phi))$ is also $\sgn$-dual, the only
possibility when $\dim V$ is even is the one stated in (3).  

If
$\dim V$ is odd, then
$\iota(X_{-1}(e, \phi))=\wti\sigma(u,[\phi])^++\wti\sigma(u,[\phi])^-$.
It remains to verify
that $\wti\sigma(u,[\phi])^-=\wti\sigma(u,[\phi])^+\otimes\sgn$. Suppose this is not the
case, so that $\wti\sigma(u,[\phi])^\pm$ are $\sgn$ self-dual. By
(\ref{e:upper-triang-1}), they occur in $\sigma(e, \phi)\otimes \C S$
with multiplicity $1$. But $\C S=S^++S^-$ and since
$S^+=S^-\otimes\sgn$, this is impossible.

Claim (3) follows similarly from Proposition \ref{p:iota}.

\end{proof}

\subsection{}\label{sec:char-formula}
As mentioned in the introduction, Theorem \ref{t:main-spin} can be used to obtain character formulas for $X_{-1}(e, \phi)$ and $H^*(\C B_e)^\phi.$ Recall the notation from the introduction $$\wti\Sigma(e, \phi)=\iota(X_{-1}(e, \phi))=X_{-1}(e, \phi)\otimes\C S.$$
For every $w\in W_{\text{(-1)-ell}}$, we have then
\begin{equation}\label{e:char-form-1}
\tr(w,X_{-1}(e, \phi))=\sum_{i=0}^{d_e}(-1)^{d_e-i}\tr(w,H^{2i}(\C
B_e)^\phi)=\frac{\tr(\wti
  w,\wti\Sigma(e, \phi))}{\tr(\wti w,\C S)},
\end{equation}
where $\wti w$ is a representative of the preimage of $w$ in $\wti W$, and $d_e=\dim \C B_e$. In particular, when $w=1$, we find
\begin{equation}\label{alt-dim}
\sum_{i=0}^{d_e}(-1)^i\dim H^{2i}(\C B_e)^\phi=(-1)^{d_e}\frac {\dim \wti\Sigma(e, \phi)}{\dim \C S}.
\end{equation}

Using Corollary
\ref{c:action-1}, (\ref{e:char-form-1}) can also be interpreted as a character
formula of $H^*(\C B_e)^\phi$ on $\delta$-twisted elliptic
conjugacy classes. 

\begin{corollary}\label{c:char-formula}
Let $w$ be a $\delta$-elliptic element of $W$. Then 
$$\tr(w\delta, H^*(\C B_e)^\phi)=(-1)^{d_e}
\sgn(w_0)\frac{\tr(\wti
  w\wti w_0,\wti\Sigma(e, \phi))}{\tr(\wti w\wti w_0,\C S)}, 
$$
where $\wti w\wti w_0$ is a representative of the preimage of $ww_0$ in $\wti W$.
\end{corollary}

The  calculations in Appendix \ref{a:component} allow us to describe $\wti\Sigma(e, \phi)$ explicitly and by comparison to \cite{C}, to identify $\wti\Sigma(e, \phi)$ in terms of the known classifications of irreducible $\wti W$-representations. Therefore, Corollary \ref{c:char-formula} can be effectively used as a character formula for $H^*(\C B_e)^\phi$ on $\delta$-elliptic classes.

\begin{example}
 In $GL(n)$, the class of the element $e \in \C N^\sol$
  is parameterized, via the Jordan canonical form, by a partition
  $\lambda$ of $n$ into distinct parts, and $A(e)=\{1\}$. A partition $\lambda$ of $n$ is called even if $\ell(\lambda)\equiv n$ (mod $2$), otherwise it is called odd, where $\ell(\lambda)$ is the number of parts of $\lambda.$ Then
  $$\wti\Sigma(u_\lambda)=a_\lambda\begin{cases}\wti\sigma_\lambda,&\text{
    $\lambda$ even},\\\wti\sigma_\lambda^++\wti\sigma_\lambda^-,&\text{ 
    $\lambda$ odd },\end{cases}$$ where
  $a_\lambda$ is as in Proposition \ref{p:typeA}, and $\sigma_\lambda$, $\sigma_\lambda^\pm$ are the irreducible $\wti
  S_n$-representations constructed by Schur (cf. \cite{St}). The dimension of $\wti\sigma_\lambda$,  or of each of $\wti\sigma_\lambda^\pm$, $\lambda=(\lambda_1,\dots,\lambda_\ell),$ is given by the formula
$$2^{\lfloor\frac{n-\ell(\lambda)}2\rfloor}g^\lambda,\quad \text{where } g^\lambda=\frac {n!}{\lambda_1!\dots\lambda_\ell!}~\prod_{i<j}\frac{\lambda_i-\lambda_j}{\lambda_i+\lambda_j}.$$
Since $\dim \C S=2^{\lfloor\frac{n}2\rfloor}$, (\ref{alt-dim}) becomes in this case:
$$\sum_{i=0}^{d_{e_\lambda}}(-1)^i\dim H^{2i}(\C B_{e_\lambda})=(-1)^{d_{e_\lambda}}~ g^\lambda.$$
\end{example}

\subsection{}

We end with an application to the decomposition of tensor products $\sigma\otimes\C S$,
$\sigma\in \widehat W.$

Let $\sigma$ be an irreducible $W$-representation. By Springer's
correspondence, write $\sigma=\sigma(e, \phi)$ for $e \in \C N,$
$\phi\in\widehat{A(e)}_0.$ Since every $\wti\sigma$ occurs in a
$\wti\Sigma(e',\phi')=\iota(X_{-1}(e',\phi'))$, with $e'\in \C N^\sol,$ we consider:
\begin{equation}
\langle\sigma(e, \phi)\otimes\C S,\iota(X_{-1}(e',\phi'))\rangle_{\wti
  W}.
\end{equation}
By (\ref{e:upper-triang-1}),
\begin{equation}
\sigma(e, \phi)\otimes \C S=\sum_{e''\ge e} K(-1)^{-1}_{(e, \phi),(e'',\phi'')}~\iota(X_{-1}(e'',\phi'')),
\end{equation}
with $K(-1)^{-1}_{(e, \phi),(e,\phi'')}=1$ if $\phi''=\phi$, and $0$
if $\phi''\neq\phi$. Using Proposition \ref{p:iota}, we find
\begin{equation}\label{e:tensor-decomp}
\langle\sigma(e, \phi)\otimes\C S,\iota(X_{-1}(e',\phi'))\rangle_{\wti
  W}=a_V\sum_{\phi''\in\widehat{A(e')}_0}
K(-1)^{-1}_{(e, \phi),(e', \phi'')}~\langle\phi',\phi''\rangle^{-1}_{A(e)}.
\end{equation}
In particular, this is zero unless $e'\ge e.$

\begin{corollary}\label{c:BMc} Let $e \in \C N$ and $\phi\in \widehat{A(e)}_0.$
If $\langle\sigma(e, \phi)\otimes\C S,\wti\sigma\rangle_{\wti W}\neq
0$, then $\wti\sigma\in\Irr_{e'}\wti W$ for some $e'\in \C N^\sol$
such that $e'\ge e.$
\end{corollary}

\appendix
\section{Component groups and spin representations}\label{a:component}
In this appendix, we use a case by case analysis to determine the
structure of the spaces $\overline R_{-1}(A(e))$, the image of the map
$\iota$, and the explicit
decompositions of representations $\iota(X_{-1}(e, \phi)).$ These
calculations can also be used to relate our realization of irreducible
$\wti W$-representations with the known case by case classifications
in \cite{Mo,Rea,St}, and also \cite{C}.

\begin{remark}\label{r:8.3}
The dimension of $R(\wti W)_\gen^\Sg$ equals $|\{\sigma\in \widehat {\wti
  W}_\gen: \sigma\cong \sigma\otimes\sgn\}|+\frac 12 |\{\sigma\in \widehat {\wti
  W}_\gen: \sigma\not\cong \sigma\otimes\sgn\}|$. In particular, the
classification of irreducible $\wti W$-representations \cite{Mo,Rea,St} gives the
following dimensions for $R(\wti W)^\Sg_\gen$:
\begin{enumerate}
\item $A_{n-1}$: the number of partitions of $n$ into distinct parts;
\item $B_n$: the number of partitions of $n$;
\item $D_{n}$, $n$ odd: $\frac 12|\{\lambda\vdash n:
  \lambda\neq\lambda^t\}|+|\{\lambda\vdash n:\lambda=\lambda^t|$;
\item $D_{n}$, $n$ even: $\frac 12|\{\lambda\vdash n:
  \lambda\neq\lambda^t\}|+2|\{\lambda\vdash n:\lambda=\lambda^t|$;
\item $G_2$: 3; $F_4$: 9; $E_6$: 9; $E_7$: 13; $E_8$: 30.
\end{enumerate}
Comparing with the dimensions in Remark \ref{r:ell-classes}, we
conclude that the map $\iota:  \overline R_{-1}(W)\to R(\wti W)_{\gen}^\Sg$ from
Proposition \ref{p:iota} is an isomorphism for all irreducible $W$,
except when $W=D_{2n}$ or $E_7.$
\end{remark}

\subsection{}
We first investigate type $A$.

\begin{lemma} \label{l:Agroup-typeA}
Suppose $G=PGL(n)$ and $e_\lambda\in \C N^\sol$ is a
  nilpotent element given in the Jordan form by the partition
  $\lambda$ of $n$ with distinct parts. Then $A(e_\lambda)=\{1\}$ and 
$$\langle\triv,\triv\rangle^{-1}_{A(e_\lambda)}=2^{\ell(\lambda)-1}.$$
\end{lemma}

\begin{proof}
Straightforward.
\end{proof}

\begin{proposition}\label{p:typeA} For every distinct partition $\lambda$ of $n$, define
$$\wti\tau_\lambda=\frac 1{a_\lambda}X_{-1}(e_\lambda)\otimes\C S
,$$
where $a_\lambda=2^{\frac{\ell(\lambda)}2}$, if both $n$ and $\lambda$ are even, and $a_\lambda=2^{\lfloor\frac {\ell(\lambda)-1}2\rfloor}$, otherwise.
Then $\wti\tau_\lambda$ is irreducible $\sgn$ self dual, if $\lambda$ is an even partition, while
$\wti\tau_\lambda=\wti\tau_\lambda^++\wti\tau_\lambda^-$ if
$\lambda$ is an odd partition, with $\wti\tau_\lambda^\pm$
irreducible $\sgn$ dual to each other.
\end{proposition}

\begin{proof}
The number of irreducible genuine $\wti S_n$-representations equals
the number of conjugacy classes of $S_n$ that split when pulled back to $\wti S_n$. A well-known result (going back to Schur, see \cite[Theorem 2.1]{St}) says that this number equals the number of partitions of $n$ into odd parts plus the number of odd partitions of $n$ into distinct parts. Denote $\DP(n)$ the set of distinct partitions of $n$ and for every $\lambda\in \DP(n)$, set $b_\lambda$ equal to $1$ or $2$ if $\lambda$ is even or odd, respectively. Thus
\begin{equation}\label{e:numbers-A}
|\widehat{(\wti S_n)}_\gen|=\sum_{\lambda\in\DP(n)}b_\lambda.
\end{equation}
By Proposition \ref{p:iota} and Lemma \ref{l:Agroup-typeA}, we see that for $\lambda\in\DP(n)$,
$$\langle\iota(X_{-1}(e_\lambda)),\iota(X_{-1}(e_\lambda))\rangle_{\wti W}=\begin{cases}a_\lambda^2,&\text{ if $\lambda$ is even},\\ 2a_\lambda^2,&\text{ if $\lambda$ is odd.}\end{cases}$$
This means that $\iota(X_{-1}(e_\lambda))$ contains at least two distinct irreducible $\wti S_n$-representations when $\lambda\in\DP(n)$ is odd. Since for $\lambda\neq\lambda'$, $\iota(X_{-1}(e_\lambda))$ and $\iota(X_{-1}(e_{\lambda'}))$ are orthogonal, the claim in the Proposition follows by comparison with (\ref{e:numbers-A}).
\end{proof}

\subsection{}Next, we prove a criterion which will cover most of the
remaining cases when $G$ is not type $A$ and $e \in \C N^\sol$, but
$e$ is not distinguished. 

\begin{lemma}\label{l:Z2}
Suppose $e \in \C N^\sol$ is such that $A(e)=(\bZ/2\bZ)^k\times
(\bZ/2\bZ)^l$, $k\neq 0$, acts on the $k$-dimensional space $V_Z$ by the
representation $\refl_k\boxtimes\triv_l$, where $\refl_k$ is the
reflection representation of $(\bZ/2\bZ)^k$ and $\triv_l$ is the
trivial representation of $(\bZ/2\bZ)^l$. Then $\dim\overline R_{-1}(A(e))=2^l$ and
$$\langle\phi,\phi\rangle^{-1}_{A(e)}=1,\text{ for all }\phi\in\widehat{A(e)}.$$
Moreover, if $[\phi_1]\neq [\phi_2]$ in $\overline R_{-1}(A(e))$, then 
$$\langle[\phi_1],[\phi_2]\rangle^{-1}_{A(e)}=0.$$
\end{lemma}

\begin{proof}
Let $\sgn^{(i)}$ denote the one dimensional $(\bZ/2\bZ)^k$-representation,
with $\sgn$ on the $i$-th position and $\triv$ everywhere else. Then
$V_Z=\oplus_{i=1}^k \sgn^{(i)}$ as an $(\bZ/2\bZ)^k$-representation. Thus
$\det_{V_Z}(1+x)\neq 0$ if and only if $\text{proj}_{(\bZ/2\bZ)^k}x=1$, hence the 
$(-1)$-elliptic element in $A(e)$ are the subgroup $(\bZ/2\bZ)^l.$ It follows that
$\dim\overline R_{-1}(A(e))=2^l$. 

For the second claim, notice that since $\phi\otimes\phi=\triv,$ we
have
$\langle\phi,\phi\rangle^{-1}_{A(e)}=\langle\triv,\triv\rangle^{-1}_{A(e)}.$
It is straightforward that
$\langle\triv,\triv\rangle^{-1}_{A(e)}=\langle\triv,\wedge V_Z\rangle_{A(e)}=1.$

For the last claim, let $[\phi_1]\neq [\phi_2]$ in $\overline R_{-1}(A(e))$. We can choose $\phi_i=\triv_k\otimes\phi_i'$, $i=1,2$, where $\triv_k$ is the trivial $(\bZ/2\bZ)^k$-representation, and $\phi_1'\neq \phi_2'$ are one dimensional representations of $(\bZ/2\bZ)^l.$ Since $\phi_1'\otimes\phi_2'\neq \triv$, it does not occur in $\wedge V_Z,$ so $\langle[\phi_1],[\phi_2]\rangle^{-1}_{A(e)}=0.$

\end{proof}

\begin{lemma}\label{l:exceptions}
Let $G$ be simple and adjoint and $e \in \C N^\sol$, but not
distinguished. Then $A(e)$ is as in Lemma \ref{l:Z2}, except when:
\begin{enumerate}
\item $G$ is of type $D_{n}$, $n$ even, and $e=e_\lambda$ corresponds via the
  Jordan form to a partition $\lambda=(a_1,a_1,a_2,a_2,\dots,a_{k},a_{k})$
  of $2n$ where $a_i$ are distinct odd positive integers. In this
  case, $A(e_\lambda)=(\bZ/2\bZ)^{k-1}$ acts on the $k$-dimensional
  space $V_Z$ by twice the reflection representation. Then $\overline
  R_{-1}(A(e_\lambda))$ is one dimensional, and $\langle
  \triv,\triv\rangle^{-1}_{A(e_\lambda)}=2.$
\item $G$ is of type $D_{n}$, $n$ odd, and $e=e_\lambda$ corresponds via the
  Jordan form to a partition $\lambda=(a_1,a_1,a_2,a_2,\dots,a_{k},a_{k})$
  of $2n$ where $a_i$ are distinct odd positive integers. In this
  case, $A(e_\lambda)=(\bZ/2\bZ)^{k-1}$ acts on the $k$-dimensional
  space $V_Z=V_Z^{A(e)}\oplus V_{Z}'$, with $\dim V_Z^{A(e)}=1,$ by
  the reflection representation on $V_Z'$. Then $\overline
  R_{-1}(A(e_\lambda))$ is one dimensional, and $\langle
  \triv,\triv\rangle^{-1}_{A(e_\lambda)}=2.$
\item $G$ is of type $E_7$ and $e$ is of type $A_4+A_1$. The component group $A(e)=\bZ/2\bZ$ acts
  on the two dimensional space $V_Z$ by twice the $\sgn$
  representation. Then $\overline
  R_{-1}(A(e)$ is one dimensional, but $\langle
  \triv,\triv\rangle^{-1}_{A(e)}=2.$
\item $G$ is of type $E_6$ and $e$ is of type $D_4(a_1).$ The component group $A(e)=S_3$ acts on
  the two dimensional space $V_Z$ by the reflection
  representation. Then $\overline R_{-1}(A(e))$ is two
  dimensional, spanned by
  $[\triv]$ and $[\refl]$, and 
\begin{equation}
\langle\triv,\triv\rangle^{-1}_{A(e)}=1,\ \langle\refl,\refl\rangle^{-1}_{A(e)}=3,\ \langle\triv,\refl\rangle^{-1}_{A(e)}=1.
\end{equation} 
\end{enumerate}
\end{lemma}

\begin{proof}
The proof is a direct calculation based on the classification of
nilpotent orbits and their component groups.
\end{proof}

\subsection{} 
Suppose $W$ is of type $B_n$ and $G=Sp(2n).$ The
nilpotent orbits $e \in \C N^\sol$ are in one to one
correspondence with partitions $\mu$ of $2n$ such that $\mu$ has only
even parts and the multiplicity of each part is at most $2$. The distinguished nilpotent
$e_\mu$ correspond to $\mu$ a partition with even distinct parts.
Denote by $\DP(2n)_\ev$ the set of distinct partitions with even
parts of $2n$, and by $\qDP(2n)_\ev$ the set of partitions with even
parts of $2n$ where every part has multiplicity at most $2$ and there
is one part with multiplicity $2$. By Remark \ref{r:8.3}, the number
of irreducible $\wti W$-representations, up to tensoring with $\sgn$,
equals $|P(n)|$, the number of partitions of $n$. 

For every $\mu\in \DP(2n)_\ev\cup
  \qDP(2n)_\ev$ and $\phi\in\widehat {A(e_\mu)}_0,$
set $$\wti\tau(e_\mu,\phi)=X_{-1}(e_\mu,\phi)\otimes\C
S.$$ Then, Lemma \ref{l:Z2} yields:

\begin{proposition}
\begin{enumerate}
\item If $n$ is even, $\wti\tau(e_\mu,\phi)$ is an irreducible $\sgn$
  self dual $\wti W$-representation.
\item If $n$ is odd,
  $\wti\tau(e_\mu,\phi)=\wti\tau(e_\mu,\phi)^++\wti\tau(e_\mu,\phi)^-$,
  where $\wti\tau(e_\mu,\phi)^\pm$ are irreducible $\wti
  W$-representations $\sgn$ dual to each other.
\end{enumerate}
\end{proposition}

\subsection{} If $W$ (and $G$) is exceptional of type $G_2$, $F_4$, or $E_8$
for every $e \in \C N$, either Theorem
\ref{t:main-spin}(2),(3) or Lemma \ref{l:Z2} applies. Since $\dim V$
is even, $a_V=1$, and
$\wti\tau(e, \phi)=\iota(X_{-1}(e, \phi))$ is an irreducible $\sgn$ self
dual $\wti W$-representation.

\subsection{} Let $W$ be of type $E_6$. There are seven orbits in
$\C N^\sol$, three of which are distinguished. From Lemma
\ref{l:exceptions}, it follows that when $u$ is of type $D_4(a_1)$, then $A(e)=S_3$ and
\begin{equation}
\wti\tau(D_4(a_1),\triv):=X_{-1}(D_4(a_1),\triv)\otimes\C S
\end{equation}
is an irreducible $\sgn$ self dual representation of $\wti W(E_6)$, while
\begin{equation}
\begin{split} 
\wti\tau(D_4(a_1),\refl): &=X_{-1}(D_4(a_1),\refl)\otimes\C S \\ &=\wti\sigma(D_4(a_1),\triv)+\wti\sigma(D_4(a_1),\refl)^++\wti\sigma(D_4(a_1),\refl)^-,
\end{split}
\end{equation}
where $\wti\sigma(D_4(a_1),\refl)^\pm$ are irreducible $\sgn$ dual $\wti W(E_6)$-representations. A basis of $\overline R_{-1}(A(e))$ consisting of orthogonal elements is $\{[\triv],[\refl]-[\triv]\}.$

\subsection{}For type $E_7$, the interesting case is the nilpotent element $e$ of type $A_4+A_1$. Then $A(e)=\bZ/2\bZ$, and $V_Z$ is two dimensional. By
Lemma \ref{l:exceptions}, 
$\langle\iota(X_{-1}(A_4+A_1,\triv)),\iota(X_{-1}(A_4+A_1,\triv))\rangle_{\wti
  W}=4$ (since $\dim V$ is odd). Using \ref{t:main-spin}(2),(3) and
Lemma \ref{l:Z2}, the classes in
$\C N^\sol\setminus\{A_4+A_1\}$ account for $11$
distinct irreducible $\wti W$-representations (modulo
$\otimes\sgn$). This implies that $\iota(X_{-1}(A_4+A_1,\triv))$ is
either two copies of a $\sgn$ self-dual irreducible representation or a sum
of two pairs of $\sgn$ dual irreducible representations. The latter is
in fact the correct one, and this can be seen either by invoking the
fact that $\wti W(E_7)$ does not have $\sgn$ self-dual irreducible
representations \cite{Mo}, or by refining the argument used to prove
Theorem \ref{t:main-spin}(2) as follows. If $
\iota(X_{-1}(A_4+A_1,\triv))=2\wti\sigma$, where $\wti\sigma$ is $\sgn$
self dual, then the only possibility is that
$\sigma(A_4+A_1,\triv)\otimes S^\pm$ each contain $\wti\sigma$ with
multiplicity $1$; moreover, there are no other $\wti
W$-representations $\wti\sigma'$ such that
$W\cdot\nu_{\wti\sigma'}=W\cdot\nu_{\wti\sigma}=W\cdot h_{A_4+A_1}.$
By Theorem \ref{c:vogan}, only $\wti\sigma$ can occur in the  
Dirac cohomology spaces, and in particular, $X_1(A_4+A_1,\triv)\otimes
(S^+-S^-)=\wti\sigma-\wti\sigma=0$ (see \cite{CT,COT} for details about
the   Dirac index). Since $(S^+-S^-)\otimes
(S^+-S^-)^*=2\wedge^{-1}V,$ it follows that $\langle
X_1(A_4+A_1,\triv),X_1(A_4+A_1,\triv)\rangle_W^{1}=0.$ But this is a
contradiction with Proposition \ref{p:6.1}(1), since $A_4+A_1$ is
quasidistinguished in $E_7.$

\subsection{} Suppose $W$ is of type $D_n$ and $G=PSO(2n)$. 

When $n$ is odd, all representatives of pairs $(e, \phi)$,  $e\in
\C N^\sol$ are as in Theorem \ref{t:main-spin}(2),(3) or as in Lemma
\ref{l:Z2}, except when $e=e_\mu$ corresponds to the partition
$\mu=(a_1,a_1,a_2,a_2,\dots, a_k,a_k)$ of $2n$, where $k$ is odd, $a_i$ are
all distinct and odd, see Lemma \ref{l:exceptions}. In this case,
since $\dim V$ is odd, we have
$$\langle\iota(X_{-1}(e_\mu,\triv)),\iota(X_{-1}(e_\mu,\triv))\rangle_{\wti
W}=4.$$ One may resolve the ambiguity in the same way as for $A_4+A_1$
in $E_7$ and find that
$\iota(X_{-1}(e_\mu,\triv))=2\wti\sigma(e_\mu,[\triv]),$ for some
irreducible, $\sgn$ self dual $\wti W$-representation
$\sigma(e_\mu,[\triv])$. (In this case, $e_\mu$ is not
quasidistinguished.) 

When $n$ is even, all representatives of pairs $(e, \phi)$,  $e\in
\C N^\sol$ are as in Theorem \ref{t:main-spin}(2),(3) or as in Lemma
\ref{l:Z2}, except when $e=e_\mu$ corresponds to the partition
$\mu=(a_1,a_1,a_2,a_2,\dots, a_k,a_k)$ of $2n$, where $k$ is even, $a_i$ are
all distinct and odd, see Lemma \ref{l:exceptions}. In this case,
since $\dim V$ is even, we have
$$\langle\iota(X_{-1}(e_\mu,\triv)),\iota(X_{-1}(e_\mu,\triv))\rangle_{\wti
W}=2,$$ so
$\iota(X_{-1}(e_\mu,\triv)=\wti\sigma(e_\mu,[\triv])_1+\wti\sigma(e_\mu,[\triv])_2.$
One can prove that $\wti\sigma(e_\mu,[\triv])_{1,2}$ are $\sgn$
self-dual \footnote{Erratum:
  In \cite[Theorem 3.8.1(2)]{C}, it is incorrectly stated that
  $(\wti\sigma)_{1,2}$ are $\sgn$ dual to each other.}
 by invoking an argument similar to that for $A_4+A_1$ in
$E_7$, using the   Dirac index in the even case
(\cite{CT,COT}) and the fact that $e_\mu$ is quasidistinguished in
$D_n$, $n$ even.

\begin{remark}\label{r:explain}
The nilpotent element $u=A_4+A_1$ in $E_7$ can be realized as a
regular nilpotent element in $Z_G(t)=A_3+A_3+A_1$, see the proof of
Proposition \ref{p:unip-cent}. Similarly, the nilpotent element
$e_\mu$ in $D_n$ with $\mu=(a_1,a_1,a_2,a_2,\dots,a_k,a_k)$, $a_i$
distinct and odd, can be realized as the pair $(e_\lambda,e_\lambda)$,
$\lambda=(a_1,a_2,\dots,a_k)$, in $Z_G(t)=D_{n/2}\times D_{n/2}$, when
$n$ is even, respectively $Z_G(t\delta)=B_{\frac{n+1}2}\times
B_{\frac{n+1}2}$ when $n$ is odd. Thus the automorphism $\theta$
coming from the symmetry of the affine Dynkin diagram interchanges the
two factors of $e_\lambda$ and
the exceptions appear to be related to this phenomenon.

\end{remark}

\section{Relation with Kostka systems}\label{sec:3}
We explain a relation between our approach and the
results of Kato \cite{Ka} about Kostka systems in the category of
$A_W$-modules, where $A_W=\bC[W]\ltimes S(V_\bC).$ The relevant
homological properties of the category of $A_W$-modules are presented
in \cite[section 2]{Ka}.

\subsection{} Retain the notation from the previous section. Thus $W$
is a finite Weyl group acting on the real reflection representation
$V$. Define 
\begin{equation}
A_W=\bC[W]\ltimes S(V_\bC);
\end{equation}
in the language of section \ref{sec:twistedirac}, this is the same as
the graded affine Hecke algebra with zero parameters.

Let $X,Y$ be $A_W$-modules. Then one may define a structure of
$A_W$-modules on $X\otimes_\bC Y$ and $\Hom_\bC[X,Y]$ as follows:
\begin{enumerate}
\item $X\otimes_\bC Y$
\begin{equation}\label{e:tensor}
\begin{aligned}
&w\cdot (x\otimes y)=w\cdot x\otimes w\cdot y,\quad w\in W,\\
&\xi\cdot (x\otimes y)=\xi\cdot x\otimes y+x\otimes\xi\cdot y,\quad
\xi\in V_\bC;
\end{aligned}
\end{equation}
\item $\Hom_\bC[X,Y]$
\begin{equation}\label{e:hom}
\begin{aligned}
&(w\cdot \phi)(x)=w\cdot \phi(w^{-1}x),\quad w\in W,\\
&(\xi\cdot\phi)(x)=\xi\cdot\phi(x)-\phi(\xi\cdot x),\quad \xi\in V_\bC,
\end{aligned}
\end{equation}
for all $\phi\in\Hom_\bC[X,Y]$.
\end{enumerate}

\begin{lemma}
\begin{enumerate}
\item Definitions (\ref{e:tensor}) and (\ref{e:hom}) extend to actions
  of $A_W$.
\item If $\bC$ is the trivial $A_W$-module ($w$ acts by $1$ and $\xi$
  acts by $0$), then $X\otimes_\bC \bC\cong X,$ as $A_W$-modules.
\item $\Hom_\bC[X,Y]\cong X^*\otimes_\bC Y.$
\end{enumerate}
\end{lemma}
\begin{proof} Straightforward.
\end{proof}

\subsection{} Since $A_W\otimes_{W} M\cong S(V)\otimes_\bC M,$ for
every $W$-module $M$, the usual Koszul complex of vector spaces admits an 
interpretation as a projective resolution of the trivial
$A_W$-module. More precisely, let
\begin{equation}
\ep: A_W\otimes_\bC \bC\to \bC,\ \ep(a\otimes \lambda)=a\cdot \lambda,
\end{equation}
i.e., the $A_W$-action in the trivial module, and
\begin{equation}
\begin{aligned}
&\partial_{n-1}: A_W\otimes_W \wedge^{n} V\to
A_W\otimes_W\wedge^{n-1}V,\\
&a\otimes(\xi_1\wedge\dots\wedge\xi_n)\mapsto \sum_{\ell=1}^n
(-1)^{\ell+1} a\xi_\ell\otimes
(\xi_1\wedge\dots\wedge\hat\xi_\ell\wedge\dots\wedge\xi_n).
\end{aligned}
\end{equation}
It is easy to check that $\partial_{n-1}$ is a well-defined
$A_W$-module homomorphism. Therefore
\begin{equation}\label{e:Koszul}
0\xleftarrow{} \bC
\xleftarrow{\ep}A_W\otimes_W\bC\xleftarrow{\partial_0} A_W\otimes_W
V\xleftarrow{\partial_1}A_W\otimes_W \wedge^2 V\xleftarrow{\partial_2}\dots
\end{equation}
is a projective resolution of the trivial module in the category of
$A_W$-modules. Since tensor products exist in this category, one may
tensor this complex by $-\otimes_\bC X$ to obtain a projective
resolution for every finite dimensional $A_W$-module $X$:
\begin{equation}\label{e:projresol}
0\xleftarrow{} X
\xleftarrow{\ep}A_W\otimes_W X\xleftarrow{\partial_0} A_W\otimes_W
V\otimes X\xleftarrow{\partial_1}A_W\otimes_W \wedge^2 V\otimes X\xleftarrow{\partial_2}\dots,
\end{equation}
where we used the isomorphism $\bC\otimes_\bC X\cong X$, and therefore
the morphism $\ep: A_W\otimes_W X\to X$ becomes the action of $A_W$ on $X$.

\subsection{} We regard $A_W$ as a graded algebra by assigning to $w$
degree $0$ and $\xi\in V_\bC$ degree $1$. This differs from the
convention in \cite{Ka}, where the elements of $V_\bC$ have degree
$2$, but it will be consistent with the results in section
\ref{sec:2}. We consider the category of graded $A_W$-modules:
$A_W$-gmod. If $X$ is a $\bZ$-graded vector space, $X=\oplus_j X(j),$
let 
$$\gdim X=\sum_{j} q^j \dim X(j)$$
denote the graded dimension of $X$. 

\begin{definition}[\cite{Ka}] If $X,Y$ are finite dimensional graded
  $A_W$-modules, define the graded Euler-Poincar\'e pairing
\begin{equation}
\langle X,Y\rangle^{\gEP}_{A_W}:=\sum_{i\ge 0} (-1)^i
\gdim\Ext^i_{A_W}(X,Y)\in \bZ[q].
\end{equation}
\end{definition}

It is clear that the maps $\partial_n$ in the Koszul complex are
graded maps of degree $0$, and this makes (\ref{e:Koszul}) and
(\ref{e:projresol}) graded complexes.

Define the graded $W$-character of $X\in A_W$-gmod:
$$\gch_W X=\sum_{\sigma\in\widehat W}\sum_{j\in \bZ} q^j\dim\Hom_W[\sigma,X(j)].$$

\begin{proposition}
If $X,Y$ are finite dimensional graded $A_W$-modules, then 
$$\langle X,Y\rangle^{\gEP}_{A_W}=\langle\gch_W X,\gch_W Y\rangle^{q}_W.$$
\end{proposition}

\begin{proof}
The proof is a simple application of the (graded) Euler-Poincar\'e principle
using the resolution (\ref{e:projresol}) and it is an immediate analogue
of the proof for the group algebra of the affine Weyl in the
non-graded setting \cite[Theorem 3.2]{OS}:
\begin{align*}
\langle X,Y\rangle^{\gEP}_{A_W}&=\sum_{i\ge 0} (-1)^i
\gdim\Ext^i_{A_W}(X,Y)=\sum_{i\ge 0} (-1)^i\gdim
H^i(\Hom_{A_W}(A_W\otimes_W\wedge^* V\otimes X,Y))\\
&=\sum_{i\ge 0} (-1)^i\gdim
\Hom_{A_W}(A_W\otimes_W\wedge^iV\otimes X,Y),\quad \text{by 
  Euler-Poincar\'e principle}\\
&=\sum_{i\ge 0} (-1)^i \gdim \Hom_W(\gch_W(X\otimes \wedge^i
V),\gch_W Y),\quad\text{by Frobenius reciprocity}\\
&=\langle X, Y\rangle^{q}_W,\quad \text{since }\gch_W\wedge^i V=q^i
\wedge^i V.
\end{align*}
\end{proof}

\begin{example}
In \cite{S-Kato}, the Springer $W$-action on $H^*(\C B_e)^\phi$ is upgraded to an action of the affine Weyl group. As a consequence, one can define a graded $A_W$-module structure $\C X_q(e,\phi)$  on $H^*(\C B_e)^\phi$ \cite{Ka}, such that $\gch_W \C X_q(e,\phi)=X_q(e,\phi).$ These are particular examples of Kostka systems \cite[Definition A]{Ka}. Thus:
$$\langle \C X_q(e,\phi),\C X_q(e',\phi')\rangle^{\gEP}_{A_W}=\langle\gch_W X_q(e,\phi),\gch_W X_q(e',\phi')\rangle^{q}_W.$$
\end{example}

\end{document}